\documentclass[12pt,a4paper,twoside]{amsart}

\numberwithin{equation}{section}
\usepackage{rotating}
\usepackage{mathrsfs}
\usepackage{amscd,amssymb,amsopn,amsmath,amsthm,graphics,amsfonts,enumerate,verbatim,calc}
\usepackage[all]{xy}
\usepackage[utf8]{inputenc}
\usepackage{color}
\newtheorem{theorem}{Theorem}[section]

\newtheorem{corollary}[theorem]{Corollary}
\newtheorem{lemma}[theorem]{Lemma}
\newtheorem{problem}[theorem]{Problem}
\newtheorem{observation}[theorem]{Observation}
\newtheorem{proposition}[theorem]{Proposition}
\newtheorem{remark}[theorem]{Remark}
\newtheorem{fact}[theorem]{Fact}

\theoremstyle{definition}
\newtheorem{definition}[theorem]{Definition}

\theoremstyle{remark}

\newtheorem{notation}[theorem]{Notation}

\newcommand{\Ker}{{\rm Ker}}

\newcommand{\rest}{{\restriction}}

\newcommand{\Rang}{{\rm Rang}}\newcommand{\Aut}{{\rm Aut}}
\newcommand{\aut}{{\rm Aut}}

\newcommand{\Prime}{{\rm Prime}}
\newcommand{\Hom}{{\rm Hom}}\newcommand{\End}{{\rm End}}

\newcommand{\Tor}{{\rm Tor}}\newcommand{\tor}{{\rm tor}}
\newcommand{\id}{{\rm id}}

\newcommand{\sn}{{\smallskip\noindent}}

\newcommand{\cH}{{\mathscr H}}

\newcommand{\cf}{{\rm cf}}

\newcount\skewfactor
\def\mathunderaccent#1#2 {\let\theaccent#1\skewfactor#2
	\mathpalette\putaccentunder}
\def\putaccentunder#1#2{\oalign{$#1#2$\crcr\hidewidth
		\vbox to.2ex{\hbox{$#1\skew\skewfactor\theaccent{}$}\vss}\hidewidth}}

\usepackage{hyperref}
\date{\today}
\begin{document}
\author[M. Asgharzadeh]{Mohsen Asgharzadeh}
\address{Hakimiyeh, Tehran, Iran.}
%\urladdr{https://shelah.logic.at/}
\email{mohsenasgharzadeh@gmail.com}

\author[M. Golshani]{Mohammad Golshani}
\address{School of Mathematics\\
 Institute for Research in Fundamental Sciences (IPM)\\
 P.O.\ Box:
19395--5746, Tehran, Iran.}
%\urladdr{https://shelah.logic.at/}
\email{golshani.m@gmail.com}

\author{Daniel Herden}
\address[]{Department of Mathematics, Baylor University, Sid Richardson Building, 1410 S.4th Street, Waco, TX 76706, USA}
\email{daniel\_herden@baylor.edu}

\author{Saharon Shelah}
\address[]{The Hebrew University of Jerusalem, Einstein Institute of Mathematics, Edmond J. Safra Campus, Givat Ram, Jerusalem 91904, Israel}
\address{Department of Mathematics, Hill Center-Busch Campus,  Rutgers, The State University of New Jersey, 110 Frelinghuysen Road, Piscataway, NJ 08854-8019, USA}
\email{shelah@math.huji.ac.il}

\thanks{}
\keywords

\title{On groups well represented as automorphism groups of groups}
\dedicatory {Dedicated to L{\'a}szl{\'o} Fuchs for his  100th  birthday}
\thanks{The second author's research has been supported by a grant from IPM (No. 1402030417). The second author's work is based upon research funded by Iran National Science Foundation (INSF) under project No. 4027168. The fourth author research partially supported by the Israel Science Foundation (ISF) grant no: 1838/19, and Israel Science Foundation (ISF) grant no: 2320/23;
	Research partially supported by the grant ``Independent Theories'' NSF-BSF, (BSF 3013005232). This is 1049 in Shelah's list of publications.}

\begin{abstract}
Assuming G\"{o}del's axiom of constructibility $\bold V=\bold L,$ we present a characterization of those groups $L$ for which there exist
arbitrarily large groups $H$ such that $\aut(H) \cong L$. In particular, we show that it suffices to have
one such group $H$ such that the  size of its center is bigger than $ 2^{|L |+\aleph_0}$.
\end{abstract}

\maketitle
%\numberwithin{equation}{section}\tableofcontents

\section{Introduction}\label{0}

The representation problem from group theory asks:

\begin{problem}\label{p1}
For a given  group $L$, is there a group $H$ such that $\aut(H)\cong L $?
\end{problem}
%The representation of $L$ by $\aut(H)=L $ indicates that $L$ is the units of a ring $R$, e.g.
%$R:=\End(H)$, which was asked, with the commutative convenience, by Fuchs some decades ago.
 There are a lot of interesting research papers in this area. Here, we recall only a short list of them. The study of automorphism groups of finite abelian groups  started with Shoda \cite{sho}.
 Hallett–Hirsch \cite{HH} and Corner \cite{CO} classified
finite groups  which are the automorphism
 group of some torsion-free group. For infinite groups, Beaumont–Pierce~\cite{BP} studied automorphism groups of torsion-free groups of rank two.
May \cite{MAY} has more results on abelian automorphism groups of torsion-free groups of countable rank.

A related problem is the classical problem of establishing whether two algebraic structures are isomorphic when their automorphism groups are isomorphic. Thanks to \cite{L5}, the answer is yes for abelian $p$-groups, when $p \neq 2$. Also, by the work of Corner–Goldsmith~\cite{COG},
the answer for reduced torsion-free modules over the ring $J_p$ of $p$-adic integers, where  $p \neq 2$,  is  affirmative.
The question of whether two abelian $2$-groups are isomorphic when their automorphism groups are isomorphic is  still an open problem.
For an excellent survey on these and related topics, see Fuchs \cite{fuchs}.

Despite  lots of work, the study of the automorphism group of groups is not documented very well, and there is no universal  solution to Problem \ref{p1}.
There are several examples of groups that can never be the automorphism group of some
abelian group, so that Problem \ref{p1} reduces to characterizing those groups $L$ that can be the automorphism
group of another group.

Throughout this paper, we work with arbitrary groups, so they are not required to be abelian. In order to study Problem
\ref{p1},  we study the structure of the automorphism group and  inner automorphisms of a given group.
It may be worth noting that Baer~\cite{B}  studied the normal subgroup structure of $\aut(H)$ for infinite $p$-groups $H$.
For more details on this and further achievements, see again the book of Fuchs \cite{fuchs}.

We suggest a good version of the problem:

\begin{problem}\label{p2}
Characterize the groups $L$ such that for arbitrarily large cardinals $\lambda$ there is a group $H$ of cardinal $\lambda$ such that $\aut(H)\cong L $.
\end{problem}

We approach Problem \ref{p2} as follows. %The organization of the paper is as follows.
Given groups $L$ and $H$ and a monomorphism
$F: L \to \aut(H)$, our aim  is to find a group $H' \supseteq H$ and an isomorphism
 $F': L \to \aut(H')$ in such a way that for each $\ell \in L,$ $F'(\ell)$ is an automorphism of $H'$ which extends $F(\ell)$.
 Thus, we have to find $H'$ in such  a way that, for each $\ell \in L,$ $F(\ell)$ extends to an automorphism of $H'$, and such that if $h \in \aut(H)$
  is not in the range of $F,$ then $h$ can not be extended to an automorphism of $H'$; furthermore, any automorphism of $H'$ must result from extending an automorphism of $H$. We show that, under some extra assumptions on our initial setup, this is always possible.

Let us recall that the group $H$ in general is not assumed to be abelian. So, we consider its center $\mathcal{Z}(H)$,
which leads to the following exact sequence:
$$\xymatrix{
	&&&& N 	\ar[dl]_{h^*}\\ &0\ar[r] &\mathcal{Z}(H)\ar[r] &H\ar[r]^{ h}&N\ar[u]_{=}\ar[r] &0
	&&&}$$
We also consider  a map $h^*$, which assigns to each $b \in N:=H/\mathcal{Z}(H)$ a preimage under $h$, but we do  not require it to be a homomorphism.
We assume that $N$ is a normal subgroup of $L$, and that it determines the inner automorphisms of $H$ in the sense that,  for $a \in H$,
if we define $a^*$ as $a^*:=F(h(a))$, then $a^*(x)=axa^{-1}$ for all $x \in H$, and
$N_{}\cong F(N)={\rm Inn}(H) \subseteq {\rm Aut}(H).$
Putting all these things together will lead to the definition of a class $\mathbf C_{\rm{aut}}$ of tuples $$\mathbf c:= (L_{\mathbf c},N_{\mathbf c},H_{\mathbf c},h_{\mathbf c},h^*_{\mathbf c},F_{\mathbf c}).$$
% equipped with  some additional properties.
% We also note that any  $f \in \Hom(N_{\bold c},\mathcal{Z}(H_{\bold c}))$  induces a natural automorphism $g_f$ of $H_{\bold c}$, defined  by   $$g_f(x) := f(h_{\mathbf c}(x)) \cdot x.$$

Our main purpose is
to furnish $\mathbf c$ with some additional structures.
This leads us to defining various versions of $\mathbf C_{\rm{aut}}$, which may serve as a framework for solving  Problem~\ref{p2}. In Section 1, we point out the right version to use. In fact, the most advanced versions of $\mathbf C_{\rm{aut}}$ which appears in this work is denoted by $\mathbf C_{\rm{aut}}^6$,  which consists of elements $({\bold c},\mathcal{G})$.
In particular, the   non-explained notion
$\mathcal{G}$ consists of a family of certain tuples ${\mathbf g}$ which assign some torsion-free abelian groups $\mathbb{G}_{\mathbf g}$. For more details,
we refer the reader to see  Definition  \ref{def3}(4).
One version of our main conclusion  is:

\begin{theorem}\label{int}
	Assume $({\bold c},\mathcal{G}) \in \mathbf C_{\rm{aut}}^6$, {\rm Hom}$(N_{\bold c},\mathbb{G}_{\mathbf g})=0$ for all $\mathbf g \in \mathcal{G}$, and  $F_{\mathbf c}: L_{\mathbf c} \stackrel{\cong}\rightarrow \aut(H_{\mathbf c})$. Let $\lambda = \text{\rm cf}(\lambda) > 2^{\|\bold c\|}$ and suppose
	$\diamondsuit_\lambda(S)$ holds for some $S \subseteq S^\lambda_{\aleph_0}$ which is stationary  and non-reflecting.
	{Then}  there is ${\mathbf m}\in\mathbf C_{\rm{aut}}^6$ such that the following holds:
	\begin{enumerate}
		
		\item [$(a)$] $H_{\mathbf m} \supseteq H_{\mathbf c}$ has size $\lambda$ and  $H_{\mathbf m}/ H_{\mathbf c}$ is $\lambda$-free.
		
		\item [$(b)$]  If $g \in \aut (H_{\mathbf m})$, then for some $b \in
		L_{\bold c}$ and $f \in \Hom(N_{\bold c},\mathcal{Z}(H_{\bold m}))$, we have
		$g = F_{\mathbf m}^b \circ g_f$, %, and vice versa, every $g \in \aut(H_{\mathbf m})$ has this form.	
where the automorphism $g_f$ of $H_{\bold m}$ is defined  by   $$g_f(x) := f(h_{\mathbf m}(x)) \cdot x.$$
	\end{enumerate}	
\end{theorem}

Here,  $S^\lambda_{\aleph_0}:=\{\alpha < \lambda\mid \cf(\alpha)=\aleph_0  \}$, and $\diamondsuit_\lambda(S)$ is the  Jensen's diamond, see Definition \ref{dia}.
Also, the mysterious condition {\rm Hom}$(N_{\bold c},\mathbb{G}_{\mathbf g})=0$ is slightly stronger
than the  trivial dual condition {\rm Hom}$(N_{\bold c},\mathbb{Z})=0$, see Remark~\ref{sl}(2). %This indicates that {\rm Hom}$(N_{\bold c},\mathbb{Z})=0$.
In fact, modulo the consistency of the existence of large cardinals, it is proved in \cite{sh:1028} that consistently $\boxtimes_\lambda$ may hold for $\lambda=\aleph_{\omega_1\cdot \omega}$, where

	\begin{enumerate}
	\item [$\boxtimes_\lambda$] If $G\neq0$ is any $\lambda$-free abelian group, then $\Hom(G,\mathbb{Z})\neq0$.
	\end{enumerate}
%In fact, by \cite{Sh:883}, we know $\boxtimes_{\aleph_{\omega_1.\omega}}$ may holds.
According to \cite{Sh:883},  $\boxtimes_{\aleph_k}$ fails for $k<\omega,$ and, by  \cite{sh:1028}, $\boxtimes_{\aleph_{\omega_1 \cdot n}}$ fails as well for all $n<\omega.$
So, we expect that  Theorem \ref{int} holds in ZFC for $\lambda < \aleph_{\omega_1\cdot \omega} $, but this result has to wait.

In Section 1, we give a systematic study of the representation framework $\mathbf C_{\rm{aut}}$. We start by defining $\mathbf C_{\rm{aut}}^+$, the class of tuples $(L,H,F)$  where $F:L\stackrel{\cong}\longrightarrow\aut(H)$. We introduce seven other versions and compare them, as  a weak version  of $\mathbf C_{\rm{aut}}^+$. Their relation to each other is as follows:
$$ \mathbf C_{\rm{aut}}^0\preceq\mathbf C_{\rm{aut}}^1\preceq\mathbf C_{\rm{aut}}^2\preceq\mathbf C_{\rm{aut}}^3\preceq\mathbf C_{\rm{aut}}^4\subseteq \mathbf C_{\rm{aut}}^5\preceq \mathbf C_{\rm{aut}}^6,$$
with the convention that $\mathbf C_{\rm{aut}}^i\preceq\mathbf C_{\rm{aut}}^{i+1}$ means $ \mathbf C_{\rm{aut}}^{i+1}$ is constructed from $\mathbf C_{\rm{aut}}^{i}$.
In
Section~2, we present a systematic study of explicit
automorphisms in $\aut(H)$. This includes
$g_f$ and its variations.
 In Section 3,  we prove Theorem \ref{int} (see Theorem \ref{5h.11}) via presenting a connection between Problem \ref{p1}
and the trivial dual conjecture, which searches for the existence of almost free groups with trivial dual, see \cite{Sh:883,sh:1028}.
Namely, we show that if $\mathbf m  \in  \mathbf M_{\mathbf c}$ is auto-rigid, i.e., $F_{\mathbf m}: L_{\mathbf m} \rightarrow \aut(H_{\mathbf m})$ is an isomorphism, and $| H_{\mathbf m} | > | L_{\mathbf c}|$, then  $\Hom(N_{\bold c},\Bbb Z)$ is trivial, see Lemma \ref{lem3}.
We then use Jensen's diamond principle 	
to complete the proof of Theorem \ref{int}. This enables us to present the following
solution to Problem \ref{p2} (see Corollary \ref{a299}):
\begin{corollary}   \label{intc}Assume G\"{o}del's axiom of constructibility $\bf V=\bf L$, and let $L$ be any group. %and let $\lambda>2^{|L |+ {\aleph_0}}$ be as in Theorem \ref{int}.
Then the following are equivalent:
\begin{itemize}	\item[$(a)$] For every  cardinal $\lambda = \text{\rm cf}(\lambda) > 2^{|L |+ {\aleph_0}}$, there is a group $H$ of cardinality  $\lambda$ such that $\aut(H)\cong L$.
	\iffalse		\item[$(b)$] there is a group $H$ of cardinality $\lambda$
	and $H'\lhd H$ of cardinality  $\leq|L |+2^{\aleph_0}$ such that $\aut(H)\cong L$ and $\frac{H}{H'}$ is $\lambda$-free.
	\fi	
	\item[$(b)$]
	There is 	$\mathbf c\in\mathbf C^+_{\rm{aut}}$
	such that $L_ \mathbf c\cong L$, $\aut(H_{\mathbf c})\cong L$ and $|\mathcal{Z}(H_{\mathbf c})| > 2^{|L |+\aleph_0}$.
	\iffalse
	\item[$(d)$] Similar to (c) for $\mathbf d$
	$\mathbb{Q}_{\mathbb{P}}$-module where $\mathbb{P}:=\mathbb{P}_{(L,H,F)}$.\fi
\end{itemize}
%2) Concerning the situation of Theorem \ref{5h.1b} (resp. Theorem \ref{5h.1c}), there is a corresponding statement as in part (1).
\end{corollary}

\section{Finding the  right framework}\label{1}

In this section, we  introduce several classes of  objects which will serve as a formal
framework for realizing a fixed group $L$ as an automorphism group.
%The realization of a fixed monoid $L$ as endomorphism monoid can be handled on a
%mutatis mutandis basis and we will only mention this case in passing, whenever there are some
%noteworthy changes.
\begin{notation}
	For a group $L$, by $e_L$ we mean the unit element. We denote  the group operation by $\cdot$, so that for two elements $\ell_1, \ell_2 \in L$, their product is denoted  by $\ell_1\cdot \ell_2$ or simply $\ell_1\ell_2$. By $\ell^{-1}$ we mean the inverse of $\ell\in L$. As usual, if $L$ is abelian we use the additive notation $(L,+,-,0)$. By a $p$-group, where $p$ is a prime number, we mean an abelian $p$-group, i.e., a group $G$
such that, for all $x \in G$, there exists some $n\ge 0$ with $p^n x=0$.
\end{notation}
\begin{notation}Let $H$ be a group which is not necessarily abelian.
\begin{enumerate}
\item[(1)] By $\mathcal{Z}(H)$  we mean the center of  $H$.
\item[(2)] Suppose $a \in H$. This induces an inner automorphism $a^*\in{\rm Inn}(H)$
defined by $a^*(x):=axa^{-1}$ for all $x \in H$. Thus ${\rm Inn}(H)=\{a^*\mid a \in H  \} \subseteq {\rm Aut}(H)$.
\item[(3)]The notation $\tor(H) $ stands for the full torsion subgroup
of $\mathcal{Z}(H)$.	
\item[(4)] A  group $G$ is pure in an abelian group $H$ if $G\subseteq H$ and
$nG = nH\cap G$ for every $n\in\mathbb{Z}$. The common notation for this notion
is $G\subseteq_\ast H$.\end{enumerate}
\end{notation}

We will pay respect to the fact that the center $\mathcal{Z}(G)$ of a group $G$
is abelian by giving preference to the additive notation on this specific subgroup of $G$.

\begin{definition}\label{a20}
Let $\mathbf C_{\rm{aut}}^+$ be the class of all tuples $\mathbf c= (L_{\mathbf c},H_{\mathbf c}, F_{\mathbf c})$ such that  $L_{\mathbf c}, H_{\mathbf c}$ are  groups and    $F_{\mathbf c}:L_{\mathbf c} \to {\rm Aut}(H_{\mathbf c})$ is an isomorphism.
Let also $\|\mathbf c\| := |L_{\mathbf c}| + |H_{\mathbf c}|$.
\end{definition}
Given a suitable group $L$, we are going to find when suitable arbitrarily large groups $H$ exist
such that for some isomorphism $F:L \to {\rm Aut}(H)$, the triple $(L, H, F)$ is in  $\mathbf C_{\rm{aut}}^+$.
For this reason, we define several new classes of objects, which in a sense weaken the above notion of $\mathbf C_{\rm{aut}}^+$ and will be better
understood in the sequel.
The main object in the following definition is the class
$\mathbf C_{\rm{aut}}^2$.
\begin{definition}\label{a2}
\begin{enumerate}
\item Let $\mathbf C_{\rm{aut}}^0$ be the class of all triples
$\mathbf c= (L_{\mathbf c},H_{\mathbf c}, F_{\mathbf c})$
such that
\begin{itemize}
\item[$(a)$]  $L_{\mathbf c}, H_{\mathbf c}$ are  groups, and

\item[$(b)$]   $F_{\mathbf c}:L_{\mathbf c} \to{\rm Aut}(H_{\mathbf c})$ is a group embedding.
\end{itemize}
We set
$F_{\mathbf c}^{\ell}:=F_{\mathbf c}({\ell}) $ for $\ell \in L_{\mathbf c}$ and  $N_{\mathbf c}:= H_{\mathbf c} / \mathcal{Z}(H_{\mathbf c}).$

\item Let $\mathbf C_{\rm{aut}}^1$ be the class of all tuples  $\mathbf c= (L_{\mathbf c},N_{\mathbf c},H_{\mathbf c},h_{\mathbf c},F_{\mathbf c},Q_{\mathbf c}^{\bar s})$ with $(L_{\mathbf c}, H_{\mathbf c}, F_{\mathbf c})\in \mathbf C_{\rm{aut}}^0$ such that:
\begin{itemize}
	\item[$(a)$]
 $N_{\mathbf c}$ is a normal subgroup of $L_{\mathbf c}$.
	\item[$(b)$]
  $h_{\mathbf c}: H_{\mathbf c} \to N_{\mathbf c}$ is an epimorphism with ${\rm Ker}(h_{\mathbf c}) = \mathcal{Z}(H_{\mathbf c})$.
	\item[$(c)$]
If $a \in H_{\mathbf c}$, then $F_{\mathbf c}(h_{\mathbf c}(a))=a^*$ is the inner automorphism of $H_{\mathbf c}$
	defined by $a^*(x)=axa^{-1}$ for all $x \in H_{\mathbf c}$, thus $N_{\mathbf c}\cong F_{\mathbf c}(N_{\mathbf c})={\rm Inn}(H_{\mathbf c}) \subseteq {\rm Aut}(H_{\mathbf c})$.
	\item[$(d)$] For all $n >0,\bar s = (s_1,\ldots,s_n) \in \mathbb Z^n$, $Q_{\mathbf c}^{\bar s}$ is an $n$-ary relation on $L_{\mathbf c}$.
	\item[$(e)$] For $\bar s \in \mathbb Z^n$, we have $(b_1,\dotsc,b_{n}) \in Q_{\mathbf c}^{\bar s}$ iff
	$\sum_{\ell=1}^{n} s_\ell F_{\mathbf c}^{b_\ell}\rest_{\mathcal{Z}(H_{\mathbf c})}=0$.
\end{itemize}

\item Let $\mathbf C_{\rm{aut}}^2$ be the class of all tuples  $\mathbf c= (L_{\mathbf c}, N_{\mathbf c}, H_{\mathbf c}, h_{\mathbf c}, h^*_{\mathbf c}, F_{\mathbf c}, Q_{\mathbf c}^{\bar s})$  such that   $(L_{\mathbf c},N_{\mathbf c}, H_{\mathbf c}, h_{\mathbf c}, F_{\mathbf c}, Q_{\mathbf c}^{\bar s})\in\mathbf C_{\rm{aut}}^1$ and
 $h^*_{\mathbf c}: N_{\mathbf c} \to H_{\mathbf c}$ is a map with $h_{\mathbf c} \circ h^*_{\mathbf c}= {\rm id}$.\footnote{
So, the map $h^*_{\mathbf c}$ assigns to each $b \in N_{\mathbf c}$ a preimage under $h_{\mathbf c}$ but may not be a homomorphism.}
\end{enumerate}
\end{definition}

\begin{definition}
\label{def2}
\begin{enumerate}
\item For $\mathbf c \in \mathbf C_{\rm{aut}}^1 \cup \mathbf C_{\rm{aut}}^2 $, let
$\text{res}_0(\mathbf c):=(L_{\mathbf c},H_{\mathbf c}, F_{\mathbf c}) \in \mathbf C_{\rm{aut}}^0$.

\item  For $\mathbf c \in \mathbf C_{\rm{aut}}^2$, let
 $\text{res}_1(\mathbf c):=(L_{\mathbf c},N_{\mathbf c},H_{\mathbf c},h_{\mathbf c},F_{\mathbf c},Q_{\mathbf c}^{\bar s}) \in \mathbf C_{\rm{aut}}^1$.
\end{enumerate}
\end{definition}

\begin{notation} (1)  The \emph{$m$-power torsion subgroup} of $G$
	is
 $$\Tor_m(G):=\{g\in \mathcal{Z}(G)\mid \exists n \ge 0 \mbox{ such that }m^{n}g=0\}.$$

(2)
Let $ \mathbb{P}$ denote the set of all prime numbers and let $p\in \mathbb{P}$.
Recall that $J_p:=\widehat{\mathbb{Z}}_p$, the ring of $p$-adic integers, is the completion of $ \mathbb{Z}$ in the $p$-adic topology.
\end{notation}

\begin{fact}
We have $\Tor_m(G) = \bigoplus \{\Tor_p(G) \mid p \emph{ is a prime factor of m}\}$.
\end{fact}

\begin{proof}
	Let  $n,n'$ be factors of $m$.
	If
	$(n,n')=1$, we observe that $\Tor_n(G)\cap\Tor_{n'}(G)=0$. Since
	$\Tor_m(G) = \sum \big\{\Tor_p(G)\, \big|\, p \mbox{ prime}, p | m\big\}$, an easy induction gives $\Tor_m(G) = \bigoplus \big\{\Tor_p(G)\, \big|\,  p \mbox{ prime}, p | m\big\}$.
\end{proof}
\begin{definition}\label{a7}
(1) Let  $\ell\in\{0,1,2\}$ and $H$ be a group.
We define  $\Prime_\ell(H)$ as the following subset of the set of prime numbers:
\begin{itemize}
	\item[$(a) $]  We say $p\in \Prime_1(H)$ if and only if $\Tor_p(H)\neq 0$.
	\item[$(b)$]  We say $p\in \Prime_2(H)$ if and only if there is an embedding $(J_p,+)\hookrightarrow \mathcal{Z}(H)$.
	\item[$(c)$]$\Prime_0(H):=\Prime_1(H)\cup \Prime_2(H)$.
 \end{itemize}
(2)  We say $H$ is \emph{$\mathbb{P}_*$-divisible}, where $\mathbb{P}_*$ is a set of prime numbers, if $ \mathcal{Z}(H)$ is $p$-divisible, for all $p \notin \mathbb{P}_*.$
\end{definition}
\begin{definition}
\label{def3}
\begin{enumerate}
\item Let $\mathbf C_{\rm{aut}}^3$ be the class of all  $\mathbf c= (L_{\mathbf c},N_{\mathbf c},H_{\mathbf c},h_{\mathbf c},h^*_{\mathbf c},F_{\mathbf c},Q_{\mathbf c}^{\bar s},H^\ast_{\mathbf c},\mathbb{P}_{\mathbf c})$ such that the following properties hold:
\begin{itemize}
	\item[$(a)$]  $\text{res}_2(\mathbf c):=(L_{\mathbf c},N_{\mathbf c},H_{\mathbf c},h_{\mathbf c},h^*_{\mathbf c},F_{\mathbf c},Q_{\mathbf c}^{\bar s})\in\mathbf C_{\rm{aut}}^2$.
	\item[$(b)$] $ \mathcal{Z}(H_{\mathbf c})$ is reduced.
\item[$(c)$]	$H^*_{\mathbf c}\subseteq H_{\mathbf c}$ is a subgroup, $H_{\mathbf c}=\bigcup\{H^*_{\mathbf c}x\mid x\in\mathcal{Z}(H_{\mathbf c})\},$  and $\mathcal{Z}(H^*_{\mathbf c})=\mathcal{Z}(H_{\mathbf c})\cap H^*_{\mathbf c}$,
hence 	 $H^*_{\mathbf c}$ is a normal subgroup of $H_{\mathbf c}$.
\item[$(d)$] $\mathbb{P}_{\mathbf c}:=\Prime_0(H_{\mathbf c})$, and we have:

$(d_1)$  If $p\in\Prime_1(H_{\mathbf c})$, then $\Tor_p(H^*_{\mathbf c})\neq 0$, and

$(d_2)$ if $p\in\Prime_2(H_{\mathbf c}),$ then $(J_p,+)$ embeds into  $\mathcal{Z}(H^*_{\mathbf c})$.
\end{itemize}

\item Let $\mathbf C_{\rm{aut}}^4$ be the class of all tuples $ \mathbf c\in  \mathbf C_{\rm{aut}}^3$ such that   $H_{\mathbf c}/H^\ast_{\mathbf c}=\mathcal{Z}(H_{\mathbf c})/ \mathcal{Z}(H^*_{\mathbf c})$ is a torsion-free abelian group which is $\mathbb{P}_\mathbf c$-divisible.

\item Let $\mathbf C_{\rm{aut}}^5$
be the class of all tuples
$(\mathbf c, \mathbf g)
$
such that  $\mathbf c\in  \mathbf C_{\rm{aut}}^4$,
and $\mathbf g:=(\mathbb{G}_{\mathbf g},F_{\mathbf g}^\ell)_{\ell\in L_{\mathbf c}}$ is defined by the following:
 \begin{itemize}

\item[$(a)$]   $ \mathbb{G}_{\mathbf g}$ is a reduced torsion-free abelian group.

\item[$(b)$] $F^\ell_{\mathbf g}\in\aut(\mathbb{G}_{\mathbf g})$
such that for all $n>0, \bar{s}\in\mathbb{Z}^n$ and $(b_1,\dotsc,b_{n})\in (L_{\mathbf c})^n$:  if
 $(b_1,\dotsc,b_{n}) \in Q_{\mathbf c}^{\bar s}$, then
$\sum_{\ell=1}^n s_\ell F_{\mathbf g}^{b_\ell}\rest_{\mathbb{G}_{\mathbf g}}=0$.

 \item[$(c)$] For all primes $p$, $J_p$ does not embed into $\mathbb{G}_{\mathbf g}$.
\end{itemize}
\item Let $\mathbf C_{\rm{aut}}^6$
be the class of all tuples
$(\mathbf c, \mathcal{G})
$ with $\mathbf c\in  \mathbf C_{\rm{aut}}^4$
such that: \begin{itemize}
	
	\item[$(a)$]
 $\mathcal{G}$ is a non-empty set, and each member $\mathbf g \in \mathcal{G}$ is such that
$(\mathbf c, \mathbf g)\in \mathbf C_{\rm{aut}}^5$.

	\item[$(b)$]   If $\mathbf g\in \mathcal{G}$ and $x\in \mathbb{G}_{\mathbf g}$, then $\mathbf g\rest\text{cl}\{x\}\in\mathcal{G}$,
where $\text{cl}\{x\}$ is the smallest pure subgroup of $ \mathbb{G}_{\mathbf g}$ containing $x$ that is closed under all $F_{\bf g}^\ell$'s for $\ell \in L_{\bf c}.$

	\item[$(c)$] If $(\mathbf c, \mathbf g)\in \mathbf C_{\rm{aut}}^5$, $\mathbf g$ is embeddable into $(\mathcal{Z}(H_{\mathbf c}),F_{\mathbf c}^\ell)_{\ell\in L_{\mathbf c}}$, and there is some
$x\in \mathbb{G}_{\mathbf g}$ so that $\mathbf g=\mathbf g\rest\text{cl}\{x\}$,  then $\mathbf g$ is isomorphic to some
member of $\mathcal{G}$.

\iffalse
	\item[$(l)$]:  Assume $\mathbf g_*,\mathbf g\in  \mathcal{G}$ be so that
$\mathbb{G}_{g_*}= \mathbb{Z}z$ and
$\mathbb{G}_{\mathbf g}= \mathbb{Z}x$.
The assignment $z\mapsto x$ induces an isomorphism
$f: \mathbb{Z}z \rightarrow \mathbb{Z}x$. Here, we assume in addition that
$f(F^{\bf g_*}_{b^*_i}(z))= F^{\bf g}_{b^*_i}(f(x))$.
 \fi

\end{itemize}\item
Let   $\|(\mathbf c, \mathcal{G})\| := |L_{\mathbf c}| + |H_{\mathbf c}|+\aleph_0+\Sigma\{\|\mathbf g\|\mid \mathbf g\in \mathcal{G}\}$.
\end{enumerate}
\end{definition}

\begin{lemma}
\label{a22}The following assertions are valid:
\begin{enumerate}
\item  If  $(\mathbf c, \mathbf g):=(\mathbf c, (\mathbb{G}_{\mathbf g},F_{\mathbf g}^\ell)_{\ell\in L_{\mathbf c}}) \in \mathbf C_{\rm{aut}}^5$, then  ${\mathbf n}:=({\mathbf c} \oplus (\mathbb{G}_{\mathbf g},F_{\mathbf g}^\ell)_{\ell\in L_{\mathbf c}}) \in \mathbf \mathbf C_{\rm{aut}}^4$, where
$\mathcal{Z}(H_{\mathbf n})=  \mathcal{Z}(H_{\mathbf c})\oplus \mathbb{G}_{\mathbf g}$,
and $F_{{\mathbf n}}^\ell$ is the unique extension to an automorphism of $H_{\mathbf n}$ extending $F_{\mathbf c}^\ell \cup F^\ell_{\mathbf g}$.

\item Similarly, if $(\mathbf c, \mathcal{G}) \in \mathbf C_{\rm{aut}}^6,$ then
 ${\mathbf n}:=({\mathbf c} \oplus \bigoplus_{\mathbf g \in \mathcal{G}}(\mathbb{G}_{\mathbf g},F_{\mathbf g}^\ell)_{\ell\in L_{\mathbf c}})) \in \mathbf \mathbf C_{\rm{aut}}^4$, where
$\mathcal{Z}(H_{\mathbf n})=  \mathcal{Z}(H_{\mathbf c})\oplus \bigoplus_{\mathbf g \in \mathcal{G}} \mathbb{G}_{\mathbf g}$,
and $F_{{\mathbf n}}^\ell$ is the unique extension to an automorphism of $H_{\mathbf n}$ extending $F_{\mathbf c}^\ell \cup \bigcup_{\mathbf g \in \mathcal{G}} F^\ell_{\mathbf g}$.

\item If	$\mathbf d\in\mathbf C_{\rm{aut}}^+$,
then there is 	$\mathbf c\in\mathbf C_{\rm{aut}}^4$ with $\rm{res}_0(\mathbf c) =\mathbf d$ and $|H^\ast_{\mathbf c}| \leq (|L_\mathbf c|+\aleph_0)^{\aleph_0}$.
\end{enumerate}
	\end{lemma}
\begin{proof}
(1) It is enough to set $F_{{\mathbf n}}^\ell(a,b):=(F^\ell_{\mathbf c}(a),F^\ell_{\mathbf g}(b))$. This fits in  the following commutative diagram:
$$\xymatrix{
	& H_{\mathbf c} \ar[r]^{\subseteq}   & H_{\mathbf n}  &\mathbb{G}_{\mathbf g}\ar[l]_{\subseteq}
	\\& H_{\mathbf c} \ar[u]^{F^\ell_{\mathbf c}}\ar[r]^{\subseteq}    &H_{\mathbf n}\ar[u]^{F_{\mathbf n}^\ell}&\mathbb{G}_{\mathbf g}\ar[u]_{F_{\mathbf g}^\ell}\ar[l]_{\supseteq},
	&&}$$ and the desired claims follow easily.

(2)+(3): These are easy.
\end{proof}
\begin{definition}\label{2b.4} \mbox{}
(1):  Let $j\in\{2,3,4,5,6\}$. For any $\mathbf c \in \mathbf C_{\rm{aut}}^j$, let $\mathbf M_{\mathbf c} \subseteq \mathbf C_{\rm{aut}}^j$
be the class of all $\mathbf m \in \mathbf C_{\rm{aut}}^j$ such that the following hold:
\begin{itemize}
\item[$(a)$]  $L_{\mathbf m}=L_{\mathbf c}$ and $N_{\mathbf m}=N_{\mathbf c}$.
\item[$(b)$]    $H_{\mathbf c}\subseteq H_{\mathbf m}$ and $h_{\mathbf c} \subseteq h_{\mathbf m}$.
\item[$(c)$]  $h^*_{\mathbf m}= h^*_{\mathbf c}$.
\item[$(d)$]  $F_{\mathbf c}^{\ell}\subseteq F_{\mathbf m}^{\ell}$ for all $\ell \in L_{\mathbf c}$. Let us depict things:	$$\xymatrix{
	& H_{\mathbf m}\ar[r]^{h_{\mathbf m}}&N_{\mathbf m}\ar[r]^{h^\ast_{\mathbf m}}&H_{\mathbf m}\ar[r]^{F_{\mathbf m}^{\ell}}&H_{\mathbf m}	\\&  H_{\mathbf c}\ar[r]^{ h_{\mathbf c}}\ar[u]^{\subseteq}&N_{\mathbf c}\ar[u]^{=}\ar[r]^{h^\ast_{\mathbf c}}&H_{\mathbf c}\ar[u]^{\subseteq} \ar[r]^{F_{\mathbf c}^{\ell}}&H_{\mathbf c}\ar[u]_{\subseteq}
	&&&}$$
\item[$(e)$]  $Q_{\mathbf m}^{\bar s}= Q_{\mathbf c}^{\bar s}$ for all $\bar s \in \bigcup_{n>0} \mathbb Z^{n}$.
\item[$(f)$]  The group $H_{\mathbf m}$ is generated by the set $\mathcal{Z}(H_{\mathbf m}) \cup H_{\mathbf c}$.
\item[$(g)$]  $\mathcal{Z}(H_{\mathbf c}) = \mathcal{Z}(H_{\mathbf m}) \cap H_{\mathbf c}$.
\item[$(h)$] If $j\geq 3$ then
$\mathbb{P}_{\mathbf m}=\mathbb{P}_{\mathbf c}$ and
$H^\ast_{\mathbf m}=H^\ast_{\mathbf c}$.
\item[$(i)$]  If $j=5$, then ${\mathbf g}_{{\mathbf m}}={\mathbf g}_{{\mathbf c}}$.
\item[$(j)$]If $j=6$, then $\mathcal{G}_{{\mathbf m}}=\mathcal{G} _{{\mathbf c}}$.
\end{itemize}

(2) Let $\mathbf c \in \mathbf C_{\rm{aut}}^j$ and $\mathbf m \in \mathbf M_{\mathbf c}$.
We call $\mathbf m $ \emph{auto-rigid},
if $\operatorname{Im} (F_{\mathbf m}) = {\rm Aut}(H_{\mathbf m})$.
\end{definition}

\begin{definition}\label{2b.11} Let $j\in\{2,3,4,5,6\}$.
	For $\mathbf c \in \mathbf C_{\rm{aut}}^j$ let relation $\le_{\mathbf c}$ on $\mathbf M_{\mathbf c}$
	be defined by $\mathbf {m_1} \le_{\mathbf c} \mathbf {m_2}$ if and only if  $H_{\mathbf {m_1}}\subseteq H_{\mathbf {m_2}}$,
	$h_{\mathbf {m_1}} \subseteq h_{\mathbf {m_2}}$, and $F_{\mathbf {m_1}}^{\ell}\subseteq F_{\mathbf {m_2}}^{\ell}$ for all $\ell \in L_{\mathbf c}$.
\end{definition}
We have the following easy observations.
\begin{remark}\label{2b.12}   Let $j\in \{2,3,4,5,6\}$ and $\mathbf c \in \mathbf C_{\rm{aut}}^j$.
	\begin{itemize}
		\item[$(1)$] The pair $(\mathbf M_{\mathbf c}, \le_{\mathbf c})$ is a poset, $\mathbf c \in \mathbf M_{\mathbf c}$ and
		$\mathbf c \le_{\mathbf c} \mathbf m$ for all $\mathbf m \in \mathbf M_{\mathbf c}$.
		\item[$(2)$] If $\mathbf m_1, \mathbf m_2 \in \mathbf M_{\mathbf c}$ with $\mathbf {m_1} \le_{\mathbf c} \mathbf {m_2}$,
		then $\mathcal{Z}(H_{\mathbf m_1}) = \mathcal{Z}(H_{\mathbf m_2}) \cap H_{\mathbf m_1}$.
		\item[$(3)$] Suppose $\delta$ is a limit ordinal and $( \mathbf m_\alpha\mid \alpha < \delta      )$
		is a $\leq_{\mathbf c}$-increasing sequence from $\mathbf M_{\mathbf c}$. Then there exists $\mathbf m=\bigcup_{\alpha < \delta}\mathbf m_\alpha \in \mathbf M_{\mathbf c}$
		which is the $\le_{\mathbf c}$-least upper bound of the sequence  $( \mathbf m_\alpha\mid \alpha < \delta      )$.
	\end{itemize}
\end{remark}

\section{Some explicit automorphisms}
Let $j\in\{ 2,3,4,5,6\}$, $\mathbf c \in \mathbf C_{\rm{aut}}^j$, and take $\mathbf m \in \mathbf M_{\mathbf c}$. Also, let $f\in{\rm Hom}(N_{\mathbf c}, \mathcal{Z}(H_{\mathbf m}))$. In this section we study the induced map $g_f(x)$ and its variations.
This map plays a crucial role in our main theorem.
\begin{definition}\label{2b.13} Let $j\in\{2,3,4,5,6\}$,  $\mathbf c \in \mathbf C_{\rm{aut}}^j$, and $\mathbf m \in \mathbf M_{\mathbf c}$.
	\begin{itemize}
	\item[$(1)$] Let ${\mathcal A}_{\mathbf m}$ be the set of all homomorphisms $g \in {\rm Aut}(H_{\mathbf m})$ such that
 $g \rest_{\mathcal{Z}(H_{\mathbf m})}= \id$ and $g(x) \in x\cdot \mathcal{Z}(H_{\mathbf m})$ for all $x \in H_{\mathbf m}$.
	\item[$(2)$] Set ${\mathcal F}_{\mathbf m}:= {\rm Hom}(N_{\mathbf c}, \mathcal{Z}(H_{\mathbf m}))$.	
	\item[$(3)$] Let $x \in H_{\mathbf m}$ and  $f \in {\mathcal F}_{\mathbf m}$. The assignment $x \mapsto x  f(h_{\mathbf m}(x))$
	defines a map
  $g_f:H_{\mathbf m} \to H_{\mathbf m}$.
\end{itemize}
\end{definition}

\begin{proposition}\label{2b.16}
We have ${\mathcal A}_{\mathbf m} = \{g_f\mid f \in {\mathcal F}_{\mathbf m}\}$
for all $\mathbf c \in \mathbf C_{\rm{aut}}^j$ and $\mathbf m \in \mathbf M_{\mathbf c}$.
\end{proposition}
\begin{proof}  First suppose that $f \in {\mathcal F}_{\mathbf m}$. We show that $g_f$ is in ${\mathcal A}_{\mathbf m}$. Clearly, the map $g_f$ is a homomorphism, as $f(h_{\mathbf m}(x)) \in \mathcal{Z}(H_{\mathbf m})$ and
the following equalities
\begin{eqnarray*}
g_f(x)g_f(y)  &=&  x f(h_{\mathbf m}(x))y
f(h_{\mathbf m}(y))  = xy f(h_{\mathbf m}(x)) f(h_{\mathbf m}(y))\\ &=& xy f(h_{\mathbf m}(xy))  = g_f(xy)
\end{eqnarray*}
hold  for all $x,y \in H_{\mathbf m}$. Furthermore, for all $x\in \mathcal{Z}(H_{\mathbf m})$, we have
$h_{\mathbf m}(x)=e_{N_{\mathbf c}}$, thus $g_f(x)=x f(h_{\mathbf m}(x))=x$, and hence $g_{ f }\rest_{\mathcal{Z}(H_{\mathbf m})} = {\rm id}$.
Finally, we note that $g_f$ is an automorphism. To show this, we claim that its inverse  is $g_{-f}$:
\begin{eqnarray*}
g_{-f}(g_f(x))   &=&  g_{-f}(x f(h_{\mathbf m}(x)) = g_{-f}(x) g_{-f}(f(h_{\mathbf m}(x)))\\ &=& g_{-f}(x) f(h_{\mathbf m}(x)) = x f(h_{\mathbf m}(x))^{-1} f(h_{\mathbf m}(x))  = x.
\end{eqnarray*}
Similarly, $g_f(g_{-f}(x)) =x$ for all $x \in H_{\mathbf m}$. This proves $g_f \in {\mathcal A}_{\mathbf m}$.
To see  the reverse inclusion, we assume $g \in {\mathcal A}_{\mathbf m}$. We must show that $g=g_f$ for some $f \in {\mathcal F}_{\mathbf m}$.
Define the map $f: N_{\mathbf c}\to \mathcal{Z}(H_{\mathbf m})$ by  $f(a) := x^{-1} g(x),$ where $a= h_{\mathbf m}^*(x)$. But, this may depend to the choice of $x$.
To show this  is a well-defined  map, assume that $h_{\mathbf m}(y)=a$ for any $y \in H_{\mathbf m}$, and we need to check $f(a)=y^{-1}g(y)$. Indeed,
 $[h_{\mathbf m}(x)=h_{\mathbf m}(y)]$ implies that $ xy^{-1} \in {\rm Ker}(h_{\mathbf m})=\mathcal{Z}(H_{\mathbf m})$. Thus, in view of
Definition~\ref{2b.13}(1) we observe that $g(xy^{-1})=xy^{-1}$. In other words, $x^{-1}g(x)=y^{-1}g(y)$, i.e., the map $f$ is well-define.
In order to show $f$ is a morphism,
let $a,b \in N_{\mathbf c}$ and choose $x,y\in H_{\mathbf m}$ so that $h_{\mathbf m}(x)=a$ and $h_{\mathbf m}(y)=b$. Then $h_{\mathbf m}(xy)=ab$,
$x^{-1}g(x) \in \mathcal{Z}(H_{\mathbf m})$ and also
\begin{eqnarray*}
f(ab) = (xy)^ {-1} g(xy) = y^{-1} \cdot x^{-1}g(x)\cdot g(y)=x^{-1}g(x)\cdot y^{-1}g(y)=f(a)f(b).
\end{eqnarray*}
Consequently, $f \in {\mathcal F}_{\mathbf m}$. Clearly,  $g=g_f$. The proof is now complete.
\iffalse

[Why: By $(*)_3+ (*)_4$, $g_f$ is an automorphism of $H_{\mathbf m},$ i.e., clause (a)
of \ref{2b.13} holds. By the ``moreover'' or ``first'' in the proof
of $(*)_3$, clearly clause (b) there holds. Lastly by the proof of``second'' in
the proof of $(*)_3$, clause (c) of \ref{2b.13} holds.

why the equalities? the first by the choice of $f$ recalling $f(ab)$
the second as $g$ is an auto of $H_{\mathbf m}$; the third as the definition of $f_g$; the fourth as $f(a), f(b) \in \mathcal{Z}(H_{\mathbf m})$

$(*)_4$ $f  \in \mathcal{F}_M$

[why:  it is a function from $N_{\mathbf c}$
into $\mathcal{Z}(H_{\mathbf m})$ by the choice in $(*)_2$ and it is a homomorphism  by $(*)_3$, so by
Definition \ref{2b.13} it belongs to $\mathcal{F}_M$ indeed.]

$(*)_5$ $g_f$ is well-defined and is equal to $g$.

[why: well-defined by
Definition \ref{2b.13}(3) and by it $g_f(x)=x f(h_{\mathbf m}(x))$ for $x \in H_{\mathbf m}$, which means that calling $a=h_{\mathbf m}(x) \in N_{\mathbf c}$
we have $$g_f(x)=xf(a)=xx^{-1}g(x)=g(x),$$as claimed.
]
\fi
\end{proof}

\begin{notation}\label{3g.0} \mbox{}
The notation  $\boxtimes_{\mathbf d}:=\boxtimes_{({\mathbf c,\mathbf m,\bar s,\bar b,\pi})}$ stands for the  following hypotheses:
\begin{itemize}
\item[$(a)$] $\mathbf c \in \mathbf C_{\rm{aut}}^j$ for some $j\in\{2,3,4,5,6\}$ and $\mathbf m \in \mathbf M_{\mathbf c}$,
\item[$(b)$] $\bar s = (s_1,\ldots,s_n) \in \mathbb Z^{n}$ and
$\bar b = (b_1,\ldots,b_n) \in (L_{\mathbf c})^n$ for some $n >0$,
\item[$(c)$] $\pi \in {\rm End}(N_{\mathbf c})$.
\end{itemize}

%\item[$(2)$] Let $F_{\mathbf m}^{\bar s,\bar b}\in {\rm End}(\mathcal{Z}(H_{\mathbf m}))$ be defined by
%$F_{\mathbf m}^{\bar s,\bar b}(x) = \sum_{\ell =1}^n s_\ell F_{\mathbf m}^{b_\ell}(x)$ for all $x \in \mathcal{Z}(H_{\mathbf m})$.
%With Definition~\ref{2b.4}(1)(d) and (g) we have $F_{\mathbf m}^{\bar s,\bar b} \upharpoonright \mathcal{Z}(H_{\mathbf c}) = F_{\mathbf c}^{\bar %s,\bar b}$.
\end{notation}
\iffalse
The next two items can be considered as  generalizations of Definition~\ref{2b.13} and Proposition~\ref{2b.16}.
\fi
\begin{definition}\label{3g.1}
Let
${\mathbf d}:= ({\mathbf c,\mathbf m,\bar s,\bar b,\pi}) $ and assume $\boxtimes_{\mathbf d}$.
\begin{itemize}
	
\item[$(0)$]
	We define $h'^*_{\mathbf c,\pi} := h^*_{\mathbf c} \circ \pi$.
\item[$(1)$] Let ${\mathcal A}^{\bar s,\bar b}_{\mathbf m,\pi}$ be the set of all homomorphisms $g \in {\rm End}(H_{\mathbf m})$ such that
$g \rest_{\mathcal{Z}(H_{\mathbf m})}= F_{\mathbf m}^{\bar s,\bar b}:=\sum_{\ell =1}^n s_\ell(F_{\mathbf m}^{b_\ell}\rest_{\mathcal{Z}(H_{\mathbf m})})$,
and $g(x) \in h'^*_{\mathbf c,\pi}(h_{\mathbf m}(x))\cdot \mathcal{Z}(H_{\mathbf m})$ for all $x \in H_{\mathbf m}$.
\item[$(2)$] Let ${\mathcal F}^{\bar s,\bar b}_{\mathbf m,\pi}$ be the set of all
   functions $f:N_{\mathbf c}\to\mathcal{Z}(H_{\mathbf m})$ such that for all $a,b \in N_{\mathbf c}$ we have  $f(a)f(b) = F_{\mathbf m}^{\bar s,\bar b}(t')\cdot (t'')^{-1}\cdot f(ab),$
where $t',t'' \in \mathcal{Z}(H_{\mathbf c}) \subseteq \mathcal{Z}(H_{\mathbf m})$ are uniquely determined
by
 $t' h^*_{\mathbf c}(ab)=h^*_{\mathbf c}(a)h^*_{\mathbf c}(b)  $,  	 and $t'' h'^*_{\mathbf c,\pi}(ab)=h'^*_{\mathbf c,\pi}(a)h'^*_{\mathbf c,\pi}(b).$

\item[$(3)$] For any $f \in {\mathcal F}^{\bar s,\bar b}_{\mathbf m,\pi}$
let  $g^{\bar s,\bar b}_{\pi,f}(x):H_{\mathbf m} \to H_{\mathbf m}$ be the map defined by the assignment
$
x\mapsto F_{\mathbf m}^{\bar s,\bar b}(t)\cdot f(h_{\mathbf m}(x))\cdot h'^*_{\mathbf c,\pi}(h_{\mathbf m}(x))
$
for all $x \in H_{\mathbf m}$, where $t \in \mathcal{Z}(H_{\mathbf m})$ is uniquely determined
by $x = th^*_{\mathbf c}(h_{\mathbf m}(x))$.
\end{itemize}
\end{definition}

\begin{proposition}\label{3g.4}
For  ${\mathbf d}:= ({\mathbf c,\mathbf m,\bar s,\bar b,\pi}) $
let
$\boxtimes_{\mathbf d}$.
 Then  ${\mathcal A}^{\bar s,\bar b}_{\mathbf m,\pi} = \big\{g^{\bar s,\bar b}_{\pi,f}\mid f \in {\mathcal F}^{\bar s,\bar b}_{\mathbf m,\pi}\big\}$.
\end{proposition}
\begin{proof} First, assume $f \in {\mathcal F}^{\bar s,\bar b}_{\mathbf m,\pi}$.
We shall prove that $g^{\bar s,\bar b}_{\pi,f} \in {\mathcal A}^{\bar s,\bar b}_{\mathbf m,\pi}$.
Let $x_1,x_2 \in H_{\mathbf m}$ and choose $t_i \in \mathcal{Z}(H_{\mathbf m})$ for $i\in \{1,2\}$ such that
$x_i = t_i h^*_{\mathbf c}(h_{\mathbf m}(x_i)).
$
We observe that $h_{\mathbf m}(x_i) \in N_{\mathbf c}$, and in view of  Definition~\ref{3g.1}(2), there are  $t',t'' \in \mathcal{Z}(H_{\mathbf c})$ such that
\begin{eqnarray}
h^*_{\mathbf c}(h_{\mathbf m}(x_1))h^*_{\mathbf c}(h_{\mathbf m}(x_2)) & = & t'h^*_{\mathbf c}(h_{\mathbf m}(x_1x_2)), \label{3g.4b}\\
h'^*_{\mathbf c,\pi}(h_{\mathbf m}(x_1))h'^*_{\mathbf c,\pi}(h_{\mathbf m}(x_2)) & = & t''h'^*_{\mathbf c,\pi}(h_{\mathbf m}(x_1x_2)) \quad\quad\quad\quad\quad\quad  \label{3g.4c}\\
f(h_{\mathbf m}(x_1))f(h_{\mathbf m}(x_2)) & = & F_{\mathbf m}^{\bar s,\bar b}(t')(t'')^{-1}f(h_{\mathbf m}(x_1x_2)). \label{3g.4d}\\
x_1x_2 & \stackrel{(2.1)}= & t't_1t_2 h^*_{\mathbf c}(h_{\mathbf m}(x_1x_2)).\end{eqnarray}
 Keeping in mind that \begin{eqnarray}
{\rm Rang}(F_{\mathbf m}^{\bar s,\bar b}), {\rm Rang}(f) \subseteq \mathcal{Z}(H_{\mathbf m}) \end{eqnarray} We deduce the following identities
\begin{eqnarray*}
g^{\bar s,\bar b}_{\pi,f}(x_1)g^{\bar s,\bar b}_{\pi,f}(x_2) & \stackrel{\ref{3g.1}(3)}= & F_{\mathbf m}^{\bar s,\bar b}(t_1) f(h_{\mathbf m}(x_1)) h'^*_{\mathbf c,\pi}(h_{\mathbf m}(x_1)) \cdot
F_{\mathbf m}^{\bar s,\bar b}(t_2) f(h_{\mathbf m}(x_2)) h'^*_{\mathbf c,\pi}(h_{\mathbf m}(x_2))\\
& \stackrel{(2.5)}= & F_{\mathbf m}^{\bar s,\bar b}(t_1)F_{\mathbf m}^{\bar s,\bar b}(t_2)\cdot f(h_{\mathbf m}(x_1))f(h_{\mathbf m}(x_2))\cdot
h'^*_{\mathbf c,\pi}(h_{\mathbf m}(x_1))h'^*_{\mathbf c,\pi}(h_{\mathbf m}(x_2))\\
& \stackrel{(2.3)}= & F_{\mathbf m}^{\bar s,\bar b}(t_1t_2) \cdot F_{\mathbf m}^{\bar s,\bar b}(t')(t'')^{-1}f(h_{\mathbf m}(x_1x_2)) \cdot t''h'^*_{\mathbf c,\pi}(h_{\mathbf m}(x_1x_2))\\
& \stackrel{(2.5)}= & F_{\mathbf m}^{\bar s,\bar b}(t')F_{\mathbf m}^{\bar s,\bar b}(t_1t_2) \cdot (t'')^{-1}t'' \cdot f(h_{\mathbf m}(x_1x_2)) \cdot h'^*_{\mathbf c,\pi}(h_{\mathbf m}(x_1x_2))\\
& = & F_{\mathbf m}^{\bar s,\bar b}(t't_1t_2) f(h_{\mathbf m}(x_1x_2)) h'^*_{\mathbf c,\pi}(h_{\mathbf m}(x_1x_2)) \\
& \stackrel{(2.4)+\ref{3g.1}(3)}= &   g^{\bar s,\bar b}_{\pi,f}(x_1x_2).
\end{eqnarray*}
This shows that $g^{\bar s,\bar b}_{\pi,f}$ is a homomorphism.
Next, observe that for any $x \in H_{\mathbf m}$, we have  $h^*_{\mathbf c}(h_{\mathbf m}(x))= e_{H_{\mathbf m}}\cdot h^*_{\mathbf c}(h_{\mathbf m}(h^*_{\mathbf c}(h_{\mathbf m}(x))))$  and
\begin{eqnarray} \label{3g.4f}
g^{\bar s,\bar b}_{\pi,f}(h^*_{\mathbf c}(h_{\mathbf m}(x))) & = & F_{\mathbf m}^{\bar s,\bar b}(e_{H_{\mathbf m}})\cdot f(h_{\mathbf m}(h^*_{\mathbf c}(h_{\mathbf m}(x))))\cdot h'^*_{\mathbf c,\pi}(h_{\mathbf m}(h^*_{\mathbf c}(h_{\mathbf m}(x))))\\
& = & f(h_{\mathbf m}(x)) h'^*_{\mathbf c,\pi}(h_{\mathbf m}(x)). \nonumber
\end{eqnarray}
Let us plugging $x := th^*_{\mathbf c}(h_{\mathbf m}(x))$.
This implies that
\begin{eqnarray*}
g^{\bar s,\bar b}_{\pi,f}(t)g^{\bar s,\bar b}_{\pi,f}(h^*_{\mathbf c}(h_{\mathbf m}(x))) & = & g^{\bar s,\bar b}_{\pi,f}(th^*_{\mathbf c}(h_{\mathbf m}(x)))
 =  g^{\bar s,\bar b}_{\pi,f}(x)=  F_{\mathbf m}^{\bar s,\bar b}(t) f(h_{\mathbf m}(x)) h'^*_{\mathbf c,\pi}(h_{\mathbf m}(x))\\
& \stackrel{(2.6)}= & F_{\mathbf m}^{\bar s,\bar b}(t) g^{\bar s,\bar b}_{\pi,f}(h^*_{\mathbf c}(h_{\mathbf m}(x))).
\end{eqnarray*}
Consequently, $g^{\bar s,\bar b}_{\pi,f}(t) = F_{\mathbf m}^{\bar s,\bar b}(t)$. Letting $t$ range over all of $\mathcal{Z}(H_{\mathbf m})$, we have $g^{\bar s,\bar b}_{\pi,f}(t) = F_{\mathbf m}^{\bar s,\bar b}(t)$ for all $t \in \mathcal{Z}(H_{\mathbf m})$. This proves $g^{\bar s,\bar b}_{\pi,f} \in {\mathcal A}^{\bar s,\bar b}_{\mathbf m,\pi}$.
For the reverse inclusion,  assume $g \in {\mathcal A}^{\bar s,\bar b}_{\mathbf m,\pi}$. We must show that $g=g^{\bar s,\bar b}_{\pi,f}$ for some $f \in {\mathcal F}_{\mathbf m, \pi}^{\bar s,\bar b}$.
To this end, we define a map $f: N_{\mathbf c}\to \mathcal{Z}(H_{\mathbf m})$ as follows. For any $a \in N_{\mathbf c}$ set  $f(a) := g(h^*_{\mathbf c}(a))h'^*_{\mathbf c,\pi}(a)^{-1}.$
With Definition~\ref{3g.1}(1)  we have $f(a) \in \mathcal{Z}(H_{\mathbf m})$.
Let $a,b \in N_{\mathbf c}$ and $t',t'' \in \mathcal{Z}(H_{\mathbf c})$ be chosen as in Definition~\ref{3g.1}(2). We then have
\begin{eqnarray*}
f(a)f(b) & = & g(h^*_{\mathbf c}(a))h'^*_{\mathbf c,\pi}(a)^{-1}  g(h^*_{\mathbf c}(b))h'^*_{\mathbf c,\pi}(b)^{-1}
 =  g(h^*_{\mathbf c}(a))  g(h^*_{\mathbf c}(b))h'^*_{\mathbf c,\pi}(b)^{-1}   h'^*_{\mathbf c,\pi}(a)^{-1}\\
& = & g(h^*_{\mathbf c}(a)h^*_{\mathbf c}(b))(h'^*_{\mathbf c,\pi}(a) h'^*_{\mathbf c,\pi}(b))^{-1} =   g(t' h^*_{\mathbf c}(ab))(t'' h'^*_{\mathbf c,\pi}(ab))^{-1}\\
& = & g(t')g(h^*_{\mathbf c}(ab))\cdot h'^*_{\mathbf c,\pi}(ab)^{-1}(t'')^{-1}   =  g(t')(t'')^{-1} \cdot g(h^*_{\mathbf c}(ab)) h'^*_{\mathbf c,\pi}(ab)^{-1}\\
& = & F_{\mathbf m}^{\bar s,\bar b}(t')(t'')^{-1}\cdot f(ab).
\end{eqnarray*}
Thanks to the definition, $f \in {\mathcal F}_{\mathbf m,\pi}^{\bar s,\bar b}$ follows. Next, let $x \in H_{\mathbf m}$ and $t \in \mathcal{Z}(H_{\mathbf m})$ be chosen as in
Definition~\ref{3g.1}(3). We then have
\begin{eqnarray*}
g(x)   =  g(th^*_{\mathbf c}(h_{\mathbf m}(x)))
& = &  g(t) g(h^*_{\mathbf c}(h_{\mathbf m}(x)))\\
& = & F_{\mathbf m}^{\bar s,\bar b}(t) \cdot g(h^*_{\mathbf c}(h_{\mathbf m}(x)))h'^*_{\mathbf c,\pi}(h_{\mathbf m}(x))^{-1}
\cdot h'^*_{\mathbf c,\pi}(h_{\mathbf m}(x))\\
& = & F_{\mathbf m}^{\bar s,\bar b}(t)\cdot f(h_{\mathbf m}(x))\cdot h'^*_{\mathbf c,\pi}(h_{\mathbf m}(x))   =   g^{\bar s,\bar b}_{\pi,f}(x),
\end{eqnarray*}
which shows $g=g^{\bar s,\bar b}_{\pi,f}$. This ends the proof.
\end{proof}
\begin{proposition}\label{3g.6} \mbox{} The following assertions are valid:
\begin{itemize}
\item [$(1)$] Let $f \in {\rm Hom}(H_{\mathbf m},\mathcal{Z}(H_{\mathbf m}))$ with
${\rm Im}(f) \subseteq {\rm Ker}(f)$, and take $x \in H_{\mathbf m}$. The assignment $x\mapsto g^1_f(x) := x \cdot f(x)$ defines an automorphism $g_f^1:H_{\mathbf m} \to H_{\mathbf m}$.
\item [$(2)$] Assume $\boxtimes_{\mathbf d}$ holds with $g \in {\mathcal A}^{\bar s,\bar b}_{\mathbf m,\pi}$. Then $h\in {\mathcal A}^{\bar s,\bar b}_{\mathbf m,\pi}$ iff $h(x)=g(x)\cdot f(x)$ for all $x \in H_{\mathbf m}$ and
for some $f \in {\rm Hom}(H_{\mathbf m},\mathcal{Z}(H_{\mathbf m}))$ with $\mathcal{Z}(H_{\mathbf m})\subseteq {\rm Ker}(f)$.
\end{itemize}

\end{proposition}
\begin{proof}
(1):
The map $g_f^1$ is a homomorphism as $f(x) \in \mathcal{Z}(H_{\mathbf m})$ and
\begin{eqnarray*}
g_f^1(x)g_f^1(y) = x f(x)y f(y) = xy f(x)f(y) = xy f(xy)=g_f^1(xy)
\end{eqnarray*}
holds for all $x,y \in H_{\mathbf m}$. Since
${\rm Im}(f) \subseteq {\rm Ker}(f)$ we have
$f(f(-))= e_{H_{\mathbf m}}$.
Furthermore, $g_f^1$ is an automorphism with inverse $g_{-f}^1$:
\begin{eqnarray*}
g_f^1(g_{-f}^1(x))   &=&    g_f^1(x f(x)^{-1})
  =   x f(x)^{-1} f(x) f(f(x^{-1}))  =    x f(x)^{-1} f(x f(x)^{-1})  \\&=&   x f(x)^{-1} f(x) f(f(x^{-1}))
	 =   x f(x)^{-1} f(x)e_{H_{\mathbf m}}=  x.
\end{eqnarray*}
Similarly, $g_{-f}^1(g_f^1(x)) =x$ for all $x \in H_{\mathbf m}$.

(2): This is easy as well.
\end{proof}

We close this section by presenting  situations for which ${\mathcal A}^{\bar s,\bar b}_{\mathbf m,\pi} $ is  non-empty. %First, recall the following:

\begin{definition}
	\label{k-free}
	An abelian group $\mathbb{G}$ is called \emph{$\aleph_1$-free} if every subgroup of $\mathbb{G}$ of cardinality
	$< \aleph_1$, i.e., every countable subgroup, is free. More generally, an abelian group $\mathbb{G}$ is called \emph{$\lambda$-free} if every subgroup of $\mathbb{G}$ of cardinality
	$< \lambda$ is free.
\end{definition}

\begin{remark}\label{3g.7}
Let
${\mathbf d}:= ({\mathbf c,\mathbf m,\bar s,\bar b,\pi}) $, assume $\boxtimes_{\mathbf d}$, and let  $\mathbf m_1, \mathbf m_2 \in \mathbf M_{\mathbf c}$ such that $\mathbf m_1\le_{\mathbf c} \mathbf m_2$. Then  the following assertions hold:
\begin{itemize}
\item [$(1)$] ${\mathcal F}^{\bar s,\bar b}_{\mathbf m_1,\pi} \subseteq {\mathcal F}^{\bar s,\bar b}_{\mathbf m_2,\pi}$.
\item [$(2)$] Let $f \in {\mathcal F}^{\bar s,\bar b}_{\mathbf m_1,\pi}$ be given and define the following groups:\begin{itemize}
	\item[$(a)$]
 $ I_1 := \big\langle t',F_{\mathbf c}^{\bar s,\bar b}(t')\mid t'= h^*_{\mathbf c}(a_1)h^*_{\mathbf c}(a_2)h^*_{\mathbf c}(a_1a_2)^{-1} \mbox{for some } a_1,a_2\in N_{\mathbf c} \big\rangle_{\mathcal{Z}(H_{\mathbf c})} $
\item[$(b)$] $I_2 := \langle I_1, {\rm Im}(f) \rangle \subseteq \mathcal{Z}(H_{\mathbf m_1}).$\end{itemize}
Also, let $\varphi \in {\rm Hom}(I_2, \mathcal{Z}(H_{\mathbf m_2}))$ be such that $\varphi \upharpoonright I_1= \rm{id}$.
Then $\varphi \circ f \in {\mathcal F}^{\bar s,\bar b}_{\mathbf m_2,\pi}$.
\item [$(3)$] If ${\mathcal A}^{\bar s,\bar b}_{\mathbf m_1,\pi} \ne \emptyset$, then ${\mathcal A}^{\bar s,\bar b}_{\mathbf m_2,\pi} \ne \emptyset$.
\item [$(4)$] Suppose ${\mathcal A}^{\bar s,\bar b}_{\mathbf m_2,\pi} \ne \emptyset$ and let $\mathcal{Z}(H_{\mathbf m_2})/\mathcal{Z}(H_{\mathbf m_1})$ be an $( \aleph_1 \cdot |N_{\mathbf c}|^+)$-free
   abelian group. Then ${\mathcal A}^{\bar s,\bar b}_{\mathbf m_1,\pi} \ne \emptyset$.
\end{itemize}
\end{remark}
\begin{proof}
(1): This follows from Definition \ref{3g.1}(2) and Remark~\ref{2b.12}(2).

 (2): We just need to check Definition~\ref{3g.1}(2),
keeping in mind that $F_{\mathbf m}^{\bar s,\bar b}(t')=F_{\mathbf c}^{\bar s,\bar b}(t')$ and $F_{\mathbf c}^{\bar s,\bar b}(t'), t'' \in  I_1$.

(3): This is a consequence of (1) and Proposition~\ref{3g.4}.

 (4): Let $f\in {\mathcal F}^{\bar s,\bar b}_{\mathbf m_2,\pi}$. Then
 $|\langle \mathcal{Z}(H_{\mathbf m_1}), {\rm Im}(f) \rangle/\mathcal{Z}(H_{\mathbf m_1})| \le \aleph_0 \cdot | N_{\mathbf c}| < \aleph_1 \cdot |N_{\mathbf c}|^+.$
Thus $\langle \mathcal{Z}(H_{\mathbf m_1}), {\rm Im}(f) \rangle/\mathcal{Z}(H_{\mathbf m_1})$ is a free abelian group, and we define $\varphi$ to be the projection
onto the direct summand $\mathcal{Z}(H_{\mathbf m_1}) \subseteq \langle \mathcal{Z}(H_{\mathbf m_1}), {\rm Im}(f) \rangle$. We have $\varphi \circ f \in {\mathcal F}^{\bar s,\bar b}_{\mathbf m_2,\pi}$
with (2), and even $\varphi \circ f \in {\mathcal F}^{\bar s,\bar b}_{\mathbf m_1,\pi}$ as ${\rm Im}(\varphi) \subseteq \mathcal{Z}(H_{\mathbf m_1})$. This implies that ${\mathcal A}^{\bar s,\bar b}_{\mathbf m_1,\pi} \ne \emptyset$. In view of
 Proposition~\ref{3g.4} ${\mathcal A}^{\bar s,\bar b}_{\mathbf m_1,\pi} \ne \emptyset$.
\end{proof}

\section{The structure of large auto-rigid representations}\label{4}

In this section we present the proof of Theorem \ref{int} and Corollary \ref{intc}.
The next lemma gives conditions under which $\Hom(N_{\bold c},\Bbb Z)$ is trivial.
\begin{lemma}
	\label{lem3} Let $j\in\{2,3,4,5,6\}$.
	Let  $\mathbf c \in \mathbf C_{\rm{aut}}^j$,  and  suppose that there is some auto-rigid ${\mathbf m}\in M_{\mathbf c}$ with $\|\mathbf m  \| > \|\mathbf c  \|$, or just $|\mathcal{Z}(H_{\mathbf m})| > |L_{\mathbf c}|$. Then
	$\Hom(N_{\bold c},\Bbb Z)=0$.
\end{lemma}
\begin{proof}
	Suppose for the sake of contradiction that there is a nonzero  $f_* \in \Hom(N_{\bold c},\Bbb Z)$.
	For each $x \in \mathcal{Z}(H_{\mathbf m})$,
	the assignment
	$a\mapsto f_x(a) :=
	x ^{f_*(a)}$ defines a map $f_x \in  \Hom(N_{\bold c},\mathcal{Z}(H_{\mathbf m}))$.
	It then follows that by Lemma \ref{3g.6}, the map  $g_{f_x}:H_{\mathbf m} \to H_{\mathbf m}$
	defined by $g_{f_x}(t) = t \cdot f_x(h_{\mathbf m}(t))$ is in
	$\aut(H_{\mathbf m})$. Now,
	by the
	auto-rigidity, we have
	$L_\bold
	c\cong {\rm Im} (F_{\mathbf m}) = {\rm Aut}(H_{\mathbf m})$,
	hence
	$
	|L_\bold
	c| =|{\rm Aut}(H_{\mathbf m})| \geq |\mathcal{Z}(H_{\mathbf m})| = |H_{\mathbf m}| > |L_\bold
	c|,$
	which is impossible.
	This contradiction shows that  Hom$(N_{\bold c},\Bbb Z)$ is trivial, which gives the desired conclusion.
\end{proof}

\iffalse
\begin{observation} \label{3g.23} The mapping $g = g^*_f$ defined by $g(x)=f(x)x$ for $x \in H_{\mathbf m}$ is an automorphism of $H_{\mathbf m}$ when:
	$f$ is an homomorphism from $H_{\mathbf m}$ into $\mathcal{Z}(H_{\mathbf m})$ with range is
	included in its kernel.
\end{observation}

\begin{proof} The argument is divided into 3 steps:

	\begin{enumerate}\item[$(*)_1$]  $g$ is a homomorphism from $H_{\mathbf m}$ into $H_s$.

		Clearly $g$ maps $H_{\mathbf m}$ into $H_{\mathbf m}$.  Now for $x,y \in H_{\mathbf m}$ we have$$g(xy) =
		f(xy)(xy) = f(x)f(y)xy = (f(x)x)(f(y)y) = g(x)g(h)$$ by the choice of
		$g,f$ being a homomorphism, Rang$(f) \subseteq \mathcal{Z}(H_{\mathbf m})$ and
		the choice of $g$, respectively

		\item[$(*)_2$]  $g$ is one-to-one, i.e. Ker$(g) = \{e_{H_{\mathbf m}}\}$.

		[Why?  $x \in \text{ Rang}(f) \backslash \{e_{H_{\mathbf m}}\}$ then $f(x) =
		e_{H_{\mathbf m}}$ so $g(x) = f(x)x = x \ne e_{H_{\mathbf m}}$.]
		
		If $x \in H_{\mathbf m} \backslash \text{ Rang}(f),g(s) = f(x)x \notin \text{
			Rang}(f)$
		
		\item[$(*)_3$]  $g$ is onto.

		[why:
		Rang$(g)$ includes
		$g \upharpoonright \text{Rang}(f)$, see above and $$x
		\in H_{\mathbf m} \Rightarrow \text{ Rang}(g) \cap x \text{ Rang}(f) \ne
		\emptyset.$$  So as Range$(g)$ is a subgroup of $H_{\mathbf m}$, we are done.]\end{enumerate}
	
	The proof is now completed.
\end{proof}
\fi

\begin{definition}
	\label{18.5}
	Suppose $G$ is an abelian group and $\mathbf k= (  G,  ( F_G^\ell\mid  \ell \in L_{\mathbf c} ))$ is an expansion of $G$,
	where each $F_G^\ell$ is an automorphism of $G$.
	Let  $\mathbf c \in \mathbf C^1_{\rm{aut}}$ and ${\mathbf m} \in \mathbf M_{\mathbf c}$.
	
	\begin{enumerate}
		\item We say $\mathbf k$ is \emph{$\mathbf c$-correct} if for all $n<\omega, \bar s \in \mathbb Z^n$ and $\bar b \in(L_{\mathbf c})^n$ with $\bar b \in Q^{\bar s}_{\mathbf c}$
		holds
		$\sum_{\ell =1}^n F_G^\ell =0\in\End(G).$

		\item If $\mathbf k$ is $\mathbf c$-correct,  we define ${\mathbf n}={\mathbf m} \oplus \mathbf k \in \mathbf M_{\mathbf c}$  so that
		$\mathcal{Z}(H_{\mathbf n})=  \mathcal{Z}(H_{\mathbf m})\oplus G$,
		and  $F_{{\mathbf n}}^\ell$ is the unique extension of $F_{\mathbf m}^\ell \cup F_G^\ell$ to an automorphism of $H_{\mathbf n}$.
		
		%\item  If $G$ is an abelian group, then we set $\mathbf{E}(G)=\langle G, \langle F^G_\ell:  \ell \in L_{\mathbf c} \rangle   \rangle$
		%be the expansion of  $G$, where each $F^G_\ell$ is the identity automorphism on $G$.
		
		\item Suppose $(\mathbf c, {\mathbf g}) \in \mathbf C^5_{\rm{aut}}$ and ${\mathbf m} \in \mathbf M_{\mathbf c}$.
		\begin{enumerate}
			\item We say $\mathbf m$ is \emph{free over $(\mathbf c,  {\mathbf g})$}, if we can find $( f_\xi\mid \xi < \zeta  )$ such that $\mathcal{Z}(H_{\mathbf m}) = \bigoplus_{\xi < \zeta}f_\xi(\mathbb{G}_{\mathbf g}) \oplus \mathcal{Z}(H_{\mathbf c})$, where $f_\xi:(\mathbb G_{\mathbf g}, (F_{\mathbf g}^\ell)_{\ell \in L_{\mathbf c}}) \hookrightarrow(H_{\mathbf c},  (F_{\mathbf m}^{\ell})_{\ell \in L_{\mathbf c}}    )$ is an embedding. We assume that the embedding respects structures, which means
			$f_\xi F_{\mathbf g}^\ell = F_{\mathbf m}^\ell  f_\xi$.

			\item We say  $\mathbf m$ is \emph{$\lambda$-free over $\mathbf c$}, if for any subgroup $G'$ of $\mathcal{Z}(H_{\mathbf m})/ \mathcal{Z}(H_{\mathbf c})$ of size $< \lambda$, there is ${\mathbf m}' \in \mathbf M_{\mathbf c}$, ${\mathbf m}' \leq_{{\mathbf c}} {\mathbf m}$ which is free over
			${\mathbf c}$ such that $G' \subseteq H_{{\mathbf m}'}$.
			
			\item We say  $\mathbf m$ is \emph{strongly $\lambda$-free over $\mathbf c$}, if {\tt free} wins
			the following game for which a play lasts $\omega$ moves: in the $n$-th move, {\tt non-free} chooses $X_n \in [H_{\mathbf m}]^{< \lambda}$, {\tt free} chooses $\xi_n < \lambda$ and $( f_{n, \xi}\mid \xi < \xi_n )$,
			where each $f_{n, \xi}: \mathbb{G}_{\mathbf g} \rightarrow \mathcal{Z}(H_{\mathbf m})$ is an embedding, and
			$\sum_{m \leq n, \xi < \xi_n}f_{m, \xi}(\mathbb{G}_{\mathbf g} ) + \mathcal{Z}(H_{\mathbf m}) = \bigoplus_{m \leq n, \xi < \xi_n}f_{m, \xi}(\mathbb{G}_{\mathbf g} ) \oplus \mathcal{Z}(H_{\mathbf m}),$
			and it includes $X_n$. The {\tt free} player wins if he always has a legal move.
		\end{enumerate}
		\item Similarly, if $(\mathbf c, {\mathcal G}) \in \mathbf C^6_{\rm{aut}}$ and ${\mathbf m} \in \mathbf M_{\mathbf c}$, we say
		$\mathbf m$ is \emph{strongly $\lambda$-free over $(\mathbf c,  {\mathcal G})$}, if {\tt free} wins
		the following game for which a play lasts $\omega$ moves: in the $n$-th move, {\tt non-free} chooses $X_n \in [H_{\mathbf m}]^{< \lambda}$, {\tt free} chooses $\xi_n < \lambda$ and $( f_{n, \xi}\mid \xi < \xi_n )$,
		where each $f_{n, \xi}: {\bf g}_{n, \xi} \rightarrow \mathcal{Z}(H_{\mathbf m})$ is an embedding, and
		$\sum_{m \leq n, \xi < \xi_n}f_{m, \xi}({\bf g}_{m, \xi}) + \mathcal{Z}(H_{\mathbf m}) = \bigoplus_{m \leq n, \xi < \xi_n}f_{m, \xi}({\bf g}_{m, \xi}) \oplus \mathcal{Z}(H_{\mathbf m})$
		for some ${\bf g}_{m, \xi} \in \mathcal{G}$,
		and it includes $X_n$. The player {\tt free} wins if he always has a legal move.
	\end{enumerate}
	
\end{definition}

\begin{lemma}\label{tri}
	Suppose  $(\mathbf c, \mathbf g)
	\in \mathbf C_{\rm{aut}}^5$, {\rm Hom}$(N_{\bold c},\mathbb{G}_{\bold g})=0$,  and ${\mathbf m}$ is $|N_{\bold c}|^+$-free over  $(\mathbf c,  \mathbf g)$.
	Then  $\Hom(N_{\bold c},\mathcal{Z}(H_{\bold m}))\subseteq \Hom(N_{\bold c},H_{\bold c}).$
\end{lemma}

\begin{proof}
	Let $f \in \Hom(N_{\bold c},\mathcal{Z}(H_{\bold m}))$. Suppose for the sake of contradiction that $f \notin \Hom(N_{\bold c},H_{\bold c})$. This means that ${\rm Im}(f) \nsubseteq  H_{\bold c}$.
	In other words, the following compositions map
	$$\bar{f}:=N_{\bold c}\stackrel{f} \longrightarrow  \mathcal{Z}(H_{\bold m})\stackrel{\subseteq}\longrightarrow  H_{\bold m}\stackrel{\twoheadrightarrow}\longrightarrow   H_{\bold m} /H_{\mathbf c}$$
	is nonzero and  ${\rm Im}(\bar{f})\subseteq H_{\mathbf m} / H_{\mathbf c}$ is of size at most $|N_{\bold c}|$.
	Thus, by our assumption, we can find $( f_\xi\mid \xi < \zeta  )$ such that
	$f_\xi: \mathbb G_{\mathbf g} \rightarrow \mathcal{Z}(H_{\mathbf c})$ is an embedding and ${\rm Im}(\bar{f}) \subseteq \bigoplus_{\xi < \zeta}f_\xi(\mathbb{G}_{\mathbf g}).$
	Hence, for some $\xi < \zeta$, the natural projection $\pi_\xi: \bigoplus_{\xi < \zeta}f_\xi(\mathbb{G}_{\mathbf g}) \rightarrow f_\xi(\mathbb{G}_{\mathbf g})$
	satisfies that $\pi_\xi \circ \bar{f} \in \Hom(N_{\mathbf c}, f_\xi(\mathbb{G}_{\mathbf g}))$ is nonzero. We proved that $f_\xi \circ \pi_\xi \circ \bar{f}: N_{\mathbf c} \rightarrow \mathbb{G}_{\bold g}$
	is nonzero, a contradiction.
\end{proof}

\begin{remark}\label{sl}The above  proof shows that:
	\begin{enumerate}
		\item Suppose  $(\mathbf c, \mathbf g)
		\in \mathbf C_{\rm{aut}}^5$,	 $\mathbf m$ is $\lambda$-free over  $(\mathbf c,  \mathbf g)$, and $\mathbb{G}_{\mathbf g}$ is $\lambda$-free. Then $H_{\bold m} /H_{\mathbf c}$ is $\lambda$-free.

		\item Let $(\bold c, \mathcal{G}) \in\mathbf  C_{\rm{aut}}^6$. Then
		{\rm Hom}$(N_{\bold c},\mathbb{G}_{\bold g})=0$ for some $\bold g\in\mathcal{G}$ iff {\rm Hom}$(N_{\bold c},\mathbb{Z})=0$.
	\end{enumerate}
\end{remark}
\begin{definition}\label{dia}
	Suppose $\lambda > \aleph_0$ is regular and $S \subseteq \lambda$ is stationary.
	\begin{enumerate}
		\item The  \emph{Jensen's diamond $\diamondsuit_\lambda(S)$} asserts the existence
		of a sequence $(S_\alpha\mid \alpha \in S )$
		such that for every $X \subseteq \lambda$ the set $\{\alpha \in S\mid X \cap \alpha=S_\alpha    \}$
		is stationary. \item We use the following consequence of $\diamondsuit_\lambda(S)$: let $A=\bigcup_
		{\alpha<\lambda} A_\alpha$ and $B=\bigcup_
		{\alpha<\lambda} B_\alpha$ be
		two $\lambda$-filtrations. Then there are $\{g_\alpha\mid A_\alpha \to B_\alpha\mid \alpha<\lambda\}$
		such that, for any function $g: A \to B$, the set
		$\{\alpha\in S\mid g\rest_{A_\alpha}=g_\alpha\} $ is stationary in $\lambda$.
		\item $S$ is \emph{non-reflecting} if for any limit ordinal $\delta < \lambda$ of uncountably cofinality, the set $S \cap \delta$
		is non-stationary in $\delta.$\item We set  $S^\lambda_{\aleph_0}=\{\alpha < \lambda\mid \cf(\alpha)=\aleph_0     \}$.
	\end{enumerate}
\end{definition}

Recall   that $g_f(x) :=x \cdot f(h_{\mathbf m}(x))  $. The following is the main result of this section:

\begin{theorem}\label{5h.11}  Let $(\bold c, \mathcal{G}) \in\mathbf  C_{\rm{aut}}^6$, $\lambda = \text{\rm cf}(\lambda) > 2^{\|\bold c\|}$	
	and assume
	that:
	\begin{enumerate}
		\item[$(1)$]   {\rm Hom}$(N_{\bold c},\mathbb{G}_{\mathbf g})=0$ for all $\mathbf g \in \mathcal{G},$ and $F_{\mathbf c}: L_{\mathbf c} \rightarrow \aut(H_{\mathbf c})$ is an isomorphism.
		\item[$(2)$]  $S \subseteq S^\lambda_{\aleph_0}$ is stationary non-reflecting such that
		$\diamondsuit_\lambda(S)$ holds.
	\end{enumerate}
	{Then} there is some  ${\mathbf m}\in\mathbf C_{\rm{aut}}^6\cap \bold M_{\bold c}$  of size $\lambda$ such that the following holds:
	\begin{enumerate}
		\item [$(\alpha)$]     ${\mathbf m}$ is $\lambda$-free over $\mathbf c$.
		\iffalse
		$\gamma$-filtration
		$\langle {\mathbf m}_\alpha:\alpha < \lambda\rangle$ with each ${\mathbf m}_\alpha$ free
		and even ${\mathbf m}_\beta$ free over ${\mathbf m}_\alpha$ for $\alpha < \beta,\alpha
		\notin S$; we really need
		\sn
		\item [$(\gamma)^-$]  $H_{\mathbf m}/H_{\bold c}$ is $\lambda$-free and even
		\sn
		\item [$(\gamma)^{--}$]  $H_{\mathbf m}/H_{\bold c}$ is $\|N_{\bold
			c}\|^+$-free
		\sn
		\fi
		\item[$(\beta)$] $\Hom(N_{\bold c},\mathcal{Z}(H_{\bold m})) \subseteq \Hom(N_{\bold c},H_{\bold c})$.
		
		\item [$(\gamma)$]  If $g \in \aut (H_{\mathbf m})$, then for some $b \in
		L_{\bold c}$ and $f \in \Hom(N_{\bold c},\mathcal{Z}(H_{\bold m}))$, we have
		$g = F_{\mathbf m}^b \circ g_f$. %and vice versa, every $g \in \aut(H_{\mathbf m})$ has this form.	
	\end{enumerate}	
\end{theorem}

\begin{proof}
	Without loss of generality, $\mathcal{G}$ consists of pairwise disjoint elements.
Let  $\overline{g}:=( g_\alpha\mid \alpha \in S)$ be such that $g_\alpha: \alpha \rightarrow \alpha$ and it is a diamond sequence for $S$, in the sense that  for each $g: \lambda \rightarrow \lambda$,
	the set $\{\alpha \in S\mid g\upharpoonright \alpha=g_\alpha  \}$ is stationary in $\lambda$.
	Without loss of generality, we assume in addition  that the set of elements of $H_{\bold c}$ is an
	ordinal, which, by our assumption, is $<\lambda$.
	For $\gamma \leq \lambda$, we define the set $\Lambda_\gamma$, consisting of  sequences $\bar{\mathbf m}=( {\mathbf m}_\alpha\mid \alpha < \gamma )$
	of length $\gamma$ and a set  $U_0(\bar{\mathbf m}) \subseteq \gamma$ such that:
	\begin{enumerate}
		\item[$(\ast)_0$]
		\begin{enumerate}
			\item[(a)] $ {\mathbf m}_\alpha \in \bold M_{\bold c}$ has
			universe an ordinal less than $\lambda$,
			
			\item[(b)]	$ {\mathbf m}_0=\mathbf c$,	
			
			\item[(c)] ${\mathbf m}_\alpha$ is free over $\mathbf c$, in particular, $H_{{\mathbf m}_\alpha}/ H_{{\mathbf c}}$ is free,
			
			\item[(d)] the sequence  $( {\mathbf m}_\alpha\mid \alpha < \gamma )$ is increasing and continuous at limit ordinals, i.e., ${\mathbf m}_\alpha= \bigcup_{\beta < \alpha}{\mathbf m}_\beta$,
			for all limit ordinals $\alpha < \gamma,$
			
			\item[(e)]  if $\beta < \alpha$ and $\beta \notin S$, then
			$\mathbf m_\alpha$ is free over ${\mathbf m}_\beta$,
			
			\item[(f)]  $U_0(\bar{\mathbf m}) \subseteq \gamma$ is defined by $\delta \in U_0(\bar{\mathbf m})$ if and only if:
			
			\begin{itemize}
				\item[$(f_1)$]  $\delta \in S,$
				the set of elements of $H_{\mathbf{m}_\delta}$ is $\delta$, and
				$g_\delta\in\aut(H_{\mathbf{m}_\delta})$,
				
				\item[$(f_2)$] $g_\delta\neq(F_{\mathbf m_\delta}^b \circ g_f) \upharpoonright_{ H_{\mathbf m_\delta}},$ for any $b \in
				L_{\bold c}$ and  $f \in \Hom(N_{\bold c},\mathcal{Z}(H_{\bold c}))$,
			\end{itemize}

			\item[(g)]  if $\alpha=\beta+1$, where $\beta \in \gamma \setminus U_0(\bar{\mathbf m}),$ then $H_{\mathbf m_\alpha}$ is defined such that
			$\mathcal{Z}(H_{\mathbf m_\alpha}) = \mathcal{Z}(H_{\mathbf m_\beta}) \oplus \bigoplus_{\mathbf g \in \mathcal{G}} f_{\alpha, \mathbf g}(\mathbb G_{\mathbf g}),$
			where $f_{\alpha, \mathbf g}$ embeds $( \mathbb G_{\mathbf g}, (F_{\mathbf g}^b)_{b \in L_{\mathbf c}})$ into
			$(H_{\mathbf m_\alpha},  (F_{\mathbf m_\alpha}^{b})_{b \in L_{\mathbf c}}    )$, so that Lemma \ref{a22}(2) holds,
			
			\item[(h)] if $\alpha=\beta+1$, and $\beta \in U_0(\bar{\mathbf m})$, then for any $n<\omega$ and $ {\mathbf g}\in\mathcal{G}$ we have
			the embeddings $f^{\bf g}_{\beta, n}: \mathbb G_{\mathbf g}\rightarrow H_{\mathbf m_\alpha}$ such that:
			\begin{itemize}
				\item[$(h_1)$] $H_{\mathbf m_\alpha}	/ H_{\mathbf m_\beta}=\bigoplus\{f^{\bf g}_{\beta, n}(\mathbb{G}_{\bf g})/ H_{\mathbf m_\beta}\mid {\bf g} \in \mathcal{G}        \},$
				
				\item[$(h_2)$] $n! f^{\bf g}_{\beta, n+1}(y) - f^{\bf g}_{\beta, n}(y) \in H_{\mathbf m_\beta}\subseteq H_{\mathbf m_\alpha}$ for any
				$y \in \mathbb G_{\mathbf g}$,
				\item[$(h_3)$] the following diagram commutes:$$\xymatrix{
					& 0\ar[r]& \mathbb G_{\mathbf g}\ar[r]^{f^{\bf g}_{\beta, n}} &H_{\mathbf m_\alpha}	\\ & 0\ar[r] & \mathbb G_{\mathbf g} \ar[u]^{F_{\mathbf g}^\ell}\ar[r]^{f^{\bf g}_{\beta, n}}  &  H_{\mathbf m_\alpha}\ar[u]_{F_{\mathbf m_\alpha}^\ell}
					&&&}$$
			\end{itemize}
			\item[(i)] Suppose 	$ F_{\mathbf c}^b \circ g_f=F_{\mathbf c}^d$
			for some 	$b,d \in
			L_{\bold c}$ and $f \in \Hom(N_{\bold c},\mathcal{Z}(H_{\bold c}))$. Then
			$ F_{{\mathbf m_\alpha}}^b \circ g_f=F_{{\mathbf m_\alpha}}^d$, where $g_f(x) =x \cdot f(h_{\mathbf m}(x))$.
		\end{enumerate}	
	\end{enumerate}
	
	\begin{enumerate}	
		\item[$(\ast)_1$] Suppose $\bar{\mathbf m} \in \Lambda_\gamma.$ Let  $U(\bar{\mathbf m}) \subseteq \gamma$ be the set of all $\delta \in S \cap \gamma$ such that there are $\bar{\mathbf n}:=( \mathbf n_\alpha\mid \alpha < \lambda ), \chi, h, \mathcal{B}$ such that:
		\begin{enumerate}
			\item $\bar{\mathbf n} \in \Lambda_\lambda$, set also $\mathbf n_\lambda = \bigcup_{\alpha < \lambda} \mathbf n_\alpha$, so that $\mathbf n_\lambda \in M_{\mathbf c}$ and its universe is $\lambda,$
			
			\item $\bar{\mathbf n} \upharpoonright \delta = \bar{\mathbf m} \upharpoonright \delta,$
			
			\item $h \in \aut(H_{\mathbf n_\lambda})$,
			
			\item $h \upharpoonright \delta\neq(F_{\mathbf n_\delta}^b \circ g_f)\upharpoonright H_{\mathbf n_\delta}$, for any $b \in L_{\mathbf c}$
			and $f \in \Hom(N_{\mathbf c}, \mathcal{Z}(H_{\mathbf c}))$,
			
			\item $\chi > 2^\lambda$ is regular so that $\bar{\mathbf n}, h,  (\mathbf c, \mathcal{G}) \in \cH(\chi)$,
			
			\item $\mathcal{B} \prec (\cH(\chi), \in), ||\mathcal{B}|| < \lambda$, and $\mathcal{B} \cap \lambda=\delta,$
			
			\item $\bar{\mathbf n}, h, (\mathbf c, \mathcal{G}) \in  \mathcal{B}$ and $h \upharpoonright \delta=g_\delta.$
			
		\end{enumerate}
	\end{enumerate}
	
	\begin{enumerate}
		\item[$(\ast)_2$] Suppose  $\gamma \leq \lambda$ is an ordinal. Then easily the following assertions hold:
		\begin{enumerate}
			\item If $\gamma$ is a limit ordinal, $\bar{\mathbf m}_\alpha \in \Lambda_\alpha$ for $\alpha < \gamma$,
			and $( \bar{\mathbf m}_\alpha\mid \alpha < \gamma   )$ is $\lhd$-increasing, then $\bar{\mathbf m}_\gamma = \bigcup_{\alpha < \gamma}\bar{\mathbf m}_\alpha \in \Lambda_\gamma$,
			and it end extends all $\bar{\mathbf m}_\alpha$'s, $\alpha < \gamma.$
			
			\item If $\bar{\mathbf m} \in \Lambda_\gamma$ and $\gamma' \in [\gamma, \lambda),$ then there is $\bar{\mathbf n} \in \Lambda_{\gamma'}$
			such that $\bar{\mathbf m} \unlhd \bar{\mathbf n}$.
		\end{enumerate}
	\end{enumerate}
	[Why? For (a), it is enough to show that $\mathbf m_\gamma$ is free over $\mathbf m_\mathbf c$. To this end, let $C$ be a club of $\gamma$ which is disjoint to $S$, which exists as $S$ is non-reflecting. By clause $(\ast)_0$(c), $\mathbf m_\alpha$ is free over $\mathbf{c}$, for all $\alpha$ in $C$, and by clause $(\ast)_0$(e), $\mathbf m_\alpha$ is free over $\mathbf{m}_\beta$, for all $\beta < \alpha$ in $C$. So clearly $\mathbf m_\gamma$ is free over $\mathbf m_\mathbf c$. For clause (b), we prove something stronger in the following.]	
	
	For $\gamma \leq \lambda$ we define $\Omega_\gamma$ as the class of all  $\bar{\mathbf m} \in \Lambda_\gamma$ such that:
	\begin{enumerate}
		\item[$(\ast)_3$]
		If $\alpha=\delta+1 < \gamma$
		and $\delta \in U(\bar{\mathbf m}),$ then $g_\delta$ is unextendable, which means
		if $\bar{\mathbf  m} \upharpoonright \alpha+1 \leq_{\bf c}\bar{\mathbf   n} \in \Lambda_\beta$ with $\beta\in[\gamma,\lambda]$, then $g_\delta$ cannot be extended to an automorphism of $H_{\bar{\mathbf  n}}$.
	\end{enumerate}
	%$(\ast_2)$: Suppose $\bar{\mathbf m} \in \Lambda_\gamma$ is such that:
	%\begin{enumerate}
	%	\item[$(i)$] If $\alpha=\delta+1 < \gamma$
	%and $\delta \in U(\bar{\mathbf m})$, then there is $\bar{\mathbf  n}_1$ such that $\bar{\mathbf   m}\rest \delta\leq_{\bf c}\bar{\mathbf   n}_1\in \Lambda_\alpha$ and if $\bar{\mathbf   n}_1\leq_{\bf c}\bar{\mathbf   n}\in \Lambda_\beta$ with $\beta\in[\gamma,\lambda]$, then $g_\delta$ can no be extended to an automorphism of $H_{\bar{\mathbf n }}$, where $H_{\bar{\mathbf n }}:=\bigcup_{i} H_{{\mathbf n_i }}$
	%	\item[$(ii)$] $\bar{\mathbf   m} \upharpoonright \alpha+1$ is such that $g_\delta$ is unextendable, which means
	%if $\bar{\mathbf  m} \upharpoonright \alpha \leq_{\bf c}\bar{\mathbf   n} \in \Lambda_\beta$ with $\beta\in[\gamma,\lambda]$, then $g_\delta$ can no be extended to an automorphism of $H_{\bar{\mathbf  n}}$.
	%\end{enumerate}
	%Then $\Omega_\gamma \neq \emptyset.$
	%\begin{PROOF}{2}
	%We define $\bar{\mathbf m}^*=\langle {\mathbf m}^*_\alpha: \alpha < \gamma \rangle \in \Lambda_\gamma$ by induction on $\alpha$ such that
	%for $\alpha=\delta+1 < \gamma$,
	%if  $\delta \in U(\bar{\mathbf m})$ and if
	%\end{PROOF}

	\begin{enumerate}
		\item[$(\ast)_4$]	It is enough to prove that $\Omega_\lambda\neq \emptyset.$
	\end{enumerate}	
	[Why?	Suppose $\Omega_\lambda$ is nonempty and drive the theorem. To this end, we take $\bar{\mathbf m} \in \Omega_\lambda$ and let ${\mathbf m} = \bigcup_{\alpha < \lambda} {\mathbf m}_\alpha$. We show that
	${\mathbf m}$ is as required.
	It is clear that $\mathbf m \in M_{\mathbf c}$ and  $||\mathbf m||=\lambda.$
	Furthermore, any subgroup of $H_{\mathbf m}/ H_{\mathbf c}$ of size less than $\lambda$ is included in some
	$H_{\mathbf m_\alpha}/ H_{\mathbf c}$, for some $\alpha < \lambda$ and $H_{\mathbf m_\alpha}/ H_{\mathbf c}$ is free, hence
	$\mathbf m \in M_{\mathbf c}$
	is $\lambda$-free over $\mathbf c$.
	Since    {\rm Hom}$(N_{\bold c},\mathbb{G}_{\bold g})$ is trivial for all ${\bold g}\in \mathcal{G}$,
	by Lemma \ref{tri},
	clause $(\beta)$ of the theorem holds.
	Here, we show the clause $(\gamma)$ of the theorem holds as well.
	If  $b \in
	L_{\bold c}$ and $f \in \Hom(N_{\bold c},\mathcal{Z}(H_{\bold m}))$, then $F_{\mathbf m}^b$ is an automorphism of $H_{\mathbf m}$,
	and by Proposition \ref{2b.16}, $g_f \in \aut(H_{\mathbf m}).$  By combining these, we deduce that
	$g = F_{\mathbf m}^b \circ g_f$ is an automorphism.
	Now, let $g \in \aut(H_{\mathbf m})$. We claim that $g$ is of the form
	$g = F_{\mathbf m}^b \circ g_f$, for some  $b \in
	L_{\bold c}$ and $f \in \Hom(N_{\bold c},\mathcal{Z}(H_{\bold m}))$.
	Suppose not, so in particular,   for all $b \in
	L_{\bold c} $ and for all $ f \in \Hom(N_{\bold c},\mathcal{Z}(H_{\bold m}))$, we have:
	$g \neq F_{\mathbf m}^b \circ g_f .
	$
	As $\Hom(N_{\bold c}, \mathcal{Z}(H_{\bold m})) \subseteq \Hom(N_{\bold c},H_{\bold c})$, we have
	$|L_{\bold c}|+|\Hom(N_{\bold c},\mathcal{Z}(H_{\bold m}))| \leq |L_{\bold c}|+|\Hom(N_{\bold c},H_{\bold c})| \leq 2^{\| \mathbf c \|} < \lambda,$
	so we can find some $\alpha_*<\lambda$
	such that for all $\alpha > \alpha_*,$ and all $b, f$ as above,
	$  g \upharpoonright \mathbf m_\alpha \neq (F_{\mathbf m_\alpha}^b \circ g_f) \upharpoonright \mathbf m_\alpha.$ Choose $\chi$ large enough regular
	and let $\bar{\mathcal{B}}=( \mathcal{B}_\alpha\mid \alpha < \lambda    )$ be an increasing and continuous sequence
	of elementary submodels of $(\cH(\chi), \in)$ such that for each $\alpha <\lambda,$
	$\mathcal{B}_\alpha$ has cardinality $<\! \lambda$, $\bar{\mathbf m},  \mathbf c, g  \in \mathcal{B}_\alpha$
	and $( \mathcal{B}_\gamma\mid \gamma \leq \alpha ) \in \mathcal{B}_{\alpha+1}$.
	Set
 $C:=\{\delta > \alpha_*\mid {\mathbf m_\delta} \text{~has universe ~}\delta  \text{~and~} \mathcal{B}_\delta \cap \lambda =\delta  \}.
	$
	Then $C$ is  a club of $\lambda$, hence by the choice of the sequence $( g_\delta\mid \delta \in S )$, the set
	$S'=\{\delta\in C \cap S\mid  g \upharpoonright \delta=g_\delta  \}$ is stationary.
	Let $\delta \in S'.$ In the light of $(\ast)_2$, we conclude that  $g_\delta$ can not be extended to an isomorphism of $H_{\mathbf m}$. But, this is a contradiction, because $g \supseteq g_\delta$
	is such an extension.]
	
	In order to prove that $\Omega_\lambda\neq \emptyset$ we define $U^+({\bar{\mathbf m}}):=\big\{\delta\in U(\bar{\mathbf m})\mid $ there is $\bar{\mathbf n_1}$ such that ${\bar{\mathbf m}} \upharpoonright \delta \leq_{\mathbf c} {\bar{\mathbf n}_1} \in \Lambda_{\delta+1}$ and if ${\bar{\mathbf n_1}} \leq_{\mathbf c} {\bar{\mathbf n}} \in \Lambda_\beta$ with $\beta \in [\gamma, \lambda]$, then $g_\delta$ can not be extended to an automorphism of $H_{{\bar{\mathbf n}}}$$\big\}$. We now prove the following, which in particular, implies that $\Omega_\gamma \neq \emptyset,$
	for all $\gamma \leq \lambda$.\smallskip
	
	$(\ast)_5$ There exists $\bar{\mathbf m}$ of length $\lambda$ such that, for all $\gamma < \lambda$, $\bar{\mathbf m} \upharpoonright \gamma \in \Omega_\gamma.$\smallskip
	
	To this end, we define $\bar{\mathbf m}=( {\mathbf m}_\xi\mid \xi < \lambda ) \in \Lambda_\lambda$ by defining ${\mathbf m}_\xi$ by induction on $\xi.$
	The only non-trivial case is when $\xi=\delta+1 < \lambda$
	and $\delta \in U(\bar{\mathbf m} \upharpoonright \xi).$ Thus suppose that we are given
	$\xi$ and $\delta$ as above and $\bar{\mathbf m} \upharpoonright \delta$ is defined. We have to define ${\mathbf m}_\xi.$
For $\bar{\mathbf m} \in \Lambda_\delta,$ $\alpha \leq\beta< \delta$ and $\bf{g}\in \mathcal{G}_{\bf c}$, let  $\Upsilon$ be the family of all embedding $f: \bf{g}\hookrightarrow\mathcal{Z}(H_{{\bold m}_{\beta}})$  where ${\bold m}_{\beta} $ is $\lambda${-free} over  ${\bold m}_{\alpha}$,  $f(\mathbb{G}_{\bf g}) \cap H^*_{\bf c}=0$ and $ (f(\mathbb{G}_{\bf g})+H^*_{\bf c})/{H^*_{\bf c}} \subseteq_\ast H_{{\bold m}_{\beta}}/ H^*_{\bf c} $.  Now, we define
	\begin{enumerate}	
		\item[$(\ast)_{5.1}$]
		\begin{enumerate}
			\item[$(a)$] $\mathfrak{F}^{\bar{\mathbf m},\bf{g}}_{\alpha}:=\{f\mid f \mbox{ embeds } \bf{g} \mbox{ into }\mathcal{Z}(H_{{\bold m}_{\alpha}})  \}$.
			
			\item[$(b)$]  $\mathfrak{F}^{\bar{\mathbf m},\bf{g}}_{\alpha,\beta}:=\{f\mid f\in\Upsilon\}$.
		\end{enumerate}
		Here, by  an embedding of $\bf{g}$ into $\mathcal{Z}(H_{{\bold m}_{\beta}})$ we mean an embedding from the structure      $\bold {g}=(\mathbb{G}_{\bold g}, F_{\mathbf g}^b)_{b \in L_{\mathbf c}}$ into the structure $(\mathcal{Z}(H_{{\bold m}_{\beta}}), F_{{\mathbf m_\beta}}^b)_{b \in L_{\mathbf c}}$.
	\end{enumerate}
	To continue the proof, we split the argument into several subcases. First of all, we may assume that clause $(\ast)_0$($f_1$) holds and $\delta \in U(\bar{\mathbf m} \upharpoonright \delta)$, as otherwise, we do nothing and let $\mathbf m_\xi=\mathbf m_\delta$.
	\begin{enumerate}	
		\item[$(\ast)_{5.2}$] For $\delta \in U(\bar{\mathbf m}  \upharpoonright \delta)$, we say $\mathbf g \in \mathcal{G}$ is \emph{$1$-active with respect to $(\delta,\bar{\mathbf m} \upharpoonright \delta),$} when there are $\bar{\mathbf n}, \chi, h, \mathcal{B}$ witnessing $\delta \in U(\bar{\mathbf m}  \upharpoonright \delta)$ such that for some nonzero $x\in \mathbb{G}_{\bf{g}}$,
		for arbitrarily  large $\alpha<\delta$,  there are $\beta\in(\alpha,\delta) \setminus S$ and  $f\in\mathfrak{F}^{\bar{\mathbf m}  \upharpoonright \delta,{\bf g}}_{\alpha,\beta}$ such that
		\begin{itemize}
			\item[$(a)$]  $h(f(x))\notin \mathcal{Z}(H_{\mathbf m_\alpha})+ {\rm Im}(f)$,
			
			\item[$(b)$] $\mathcal{Z}(H_{\mathbf m_\alpha}) \cap {\rm Im}(f)=0,$
			
			\item[$(c)$] $h$ maps $H_{\mathbf m_\alpha}$ (resp.  $H_{\mathbf m_\beta}$)  into $H_{\mathbf m_\alpha}$ (resp.  $H_{\mathbf m_\beta}$).
		\end{itemize}
	\end{enumerate}

	$(\ast)_{5.3}$  If $\bf{g}$  is  1-active with respect to $(\delta,\bar{\mathbf m} \upharpoonright \delta)$, and $\xi=\delta+1,$ then for some ${\mathbf m}_\xi$,
	$\bar{\mathbf m} \upharpoonright \xi+1 \in \Lambda_{\xi+1}$ and
	$\delta\in U^+(\bar{\mathbf m} \upharpoonright \xi+1)$.
	
	[Why?	As $\cf(\delta)=\aleph_0$, there is an increasing sequence $( \alpha^0_n\mid n<\omega)\in $ $^\omega \delta$ with limit $\delta$.
	Let $\bar{\mathbf n}, \chi, h, \mathcal{B}$ witness $\delta \in U(\bar{\mathbf m}  \upharpoonright \delta)$ and let $x\in \mathbb{G}_{\bf{g}}$
	be as guaranteed in $(\ast)_{5.2}$.
	Choose $(\alpha_n, \beta_n, f_n)$ by induction on $n$ so that
	\begin{enumerate}
		\item[$(\dag)^1_n$]
		\begin{enumerate}				\item $\alpha_n\in(\alpha^0_n,\delta)\setminus S$, and $\beta_m < \alpha_n < \beta_n$ for $m<n,$
			\item If $n=m+1$ then ${\rm Im}(f_m)\subseteq\mathcal{Z}(H_{{\bold m}_{\alpha_n}}),$
			\item $(\mathbf g, x, \alpha_n,\beta_n, f_n):=(\mathbf g, x, \alpha, \beta, f)$, where    $(\mathbf g, x, \alpha, \beta, f)$ is taken  from the definition of 1-activity.
		\end{enumerate}
	\end{enumerate}
	%So, $\bigvee_{\beta\in(\alpha_n,\delta)}\Rang(f_n)\subseteqq\mathcal{Z}(H_{{\bold m}_{\alpha_n}})$.
	We are going to define $\bar{\mathbf n}_1\in\Lambda_{\delta+2}$ and $f^*_n$ such that:
	\begin{enumerate}
		\item[$(\dag)^2$]
		\begin{enumerate}
			\item $\bar{\mathbf n}_1\rest (\delta+1)=\bar{\mathbf m}\rest (\delta+1)$,
			
			\item $f^*_n$ embeds $\mathbf g = (\mathbb{G}_{\mathbf g}, (F_{\mathbf g}^b)_{b \in L_{\mathbf c}})$
			into $(H_{(\mathbf n_1)_{\delta+1}}, (F_{{(\mathbf n_1)_{\delta+1}}}^b)_{b \in L_{\mathbf c}})$,
			
			\item If $y\in{\mathbb{G}_{\bf{g}}}$  then $f^\ast_{n+1}(y)- f^\ast_n(y)=-n!f_n(y)$.
		\end{enumerate}
	\end{enumerate}
	Let $\mathcal{Z}(H_{(\mathbf m)_{\delta}})^{\widehat{}}$ denote the $\mathbb{Z}$-adic completion
	of  $\mathcal{Z}({H_{(\mathbf m)_{\delta}}})$. Note that in  $\mathcal{Z}(H_{(\mathbf m)_{\delta}})^{\widehat{}}$, $f^*_n$ is defined as
	$f^*_n(y)=\sum_{m=n}^\infty m!f_m(y).$
	Set $\xi=\sum_{n<\omega} n!$, and define $y_n, \tilde y_n \in \mathcal{Z}(H_{(\mathbf m)_{\delta}})^{\widehat{}}$ by $y_n= f_n^\ast(x)$ and $\tilde{y}_n=y_n+ \xi f_0(x)$.
	Recall that $h \restriction \delta=g_\delta$ extends to an automorphism $\hat{g_\delta}$ over $\mathcal{Z}(H_{(\mathbf m)_{\delta}})^{\widehat{}}$. We show that  either $\hat{g_\delta}(y_0)\notin   \langle \mathcal{Z}(H_{\mathbf m_{\delta}}) \cup \{ y_m\mid m<\omega\} \rangle $ or $\hat{g_\delta}(\tilde{y}_0)\notin  \langle \mathcal{Z}(H_{\mathbf m_{\delta}}) \cup \{ \tilde{y}_m\mid m<\omega\} \rangle$, and then we choose $z_0 \in \{y_0, \tilde{y}_0\}$  such that  $\hat{g_\delta}(z_0)\notin \langle \mathcal{Z}(H_{\mathbf m_{\delta}}) \cup \{ \tilde{z}_m\mid m<\omega\} \rangle$ and set
	$$\mathcal{Z}(H_{\mathbf (\mathbf n_1)_{\delta+1}}):=\langle \mathcal{Z}(H_{\mathbf m_{\delta}}) \cup \{ {z}_m\mid m<\omega\}
	\rangle\subseteq \mathcal{Z}(H_{(\mathbf m)_{\delta}})^{\widehat{}}.$$
	This easily gives us  $\bar{\mathbf n}_1\in\Lambda_{\delta+2}$, and it will be as required.
	Let us depict things:
	$$\xymatrix{
		&\mathcal{Z}(H_{\mathbf m_\delta})\ar[r]^{\subseteq}&   \mathcal{Z}(H_{\mathbf (\mathbf n_1)_{\delta+1}})\ar[r]^{\subseteq}&\mathcal{Z}(H_{(\mathbf m)_{\delta}})^{\widehat{}}	\\&  \mathcal{Z}(H_{\mathbf m_\delta})\ar[r]^{\subseteq}\ar[u]^{{g_\delta}}&   \mathcal{Z}(H_{\mathbf (\mathbf n_1)_{\delta+1}})\ar[u]^{\nexists\hat{g_\delta}\rest}\ar[r]^{\subseteq}&\mathcal{Z}(H_{(\mathbf m)_{\delta}})^{\widehat{}}\ar[u]^{\hat{g_\delta}}
		&&&}$$
	So,  suppose by way of contradiction that
	$\hat{g_\delta}(y_0)\in   \langle \mathcal{Z}(H_{\mathbf m_{\delta}}) \cup \{ y_m\mid m<\omega\} \rangle $ and $\hat{g_\delta}(\tilde{y}_0)\in  \langle \mathcal{Z}(H_{\mathbf m_{\delta}}) \cup \{ \tilde{y}_m\mid m<\omega\} \rangle$,  and we search for a contradiction.
	As $\hat{g_\delta}(y_0)\in   \langle \mathcal{Z}(H_{\mathbf m_{\delta}}) \cup \{ y_m\mid m<\omega\} \rangle $. There are $z \in \mathcal{Z}(H_{\mathbf n_{\delta}})$,  $n<\omega $ and $\{\ell_i \in \mathbb{Z}\}_{i<n}$  such that
	$
	\hat{g_\delta}(y_0) \in \mathcal{Z}(H_{\mathbf n_{\delta}}) + \sum_{i<n} \ell_i y_i.
	$
	It is easily seen that
	$\sum_{i<n} \ell_i y_i \in \mathcal{Z}(H_{\mathbf n_{\delta}}) + ( \sum_{i<n} \ell_i)y_n.$  Set
	$\ell:= \sum_{i<n} \ell_i.$
	Hence
	$\hat{g_\delta}(y_0)=\ell y_n+z$ where  $z \in \mathcal{Z}(H_{\mathbf n_{\delta}})$.
	By definition,  $\hat{g_\delta}(\tilde{y}_0)=\hat{g_\delta}(y_0+\xi f_0(x))=\ell y_n+z+\xi g_\delta  (  f_0(x)).$
	In other words,  $\hat{g_\delta}(\tilde{y}_0) -(\ell\tilde{y}_n+z)=\xi  (  g_\delta f_0(x)-\ell f_{0}(x))$.
	We now recall that
	$\langle \mathcal{Z}(H_{\mathbf m_{\delta}})  \cup \{ \tilde{y}_m\mid m<\omega\} \rangle \cap \xi \langle \mathcal{Z}(H_{\mathbf m_{\delta}})  \cup \{ \tilde{y}_m\mid m<\omega\} \rangle=0$. Finally,
	due to our assumption,  $\hat{g_\delta}(\tilde{y}_0)
	\in  \langle \mathcal{Z}(H_{\mathbf m_{\delta}}) \cup \{ \tilde{y}_m\mid m<\omega\} \rangle$,  thus we have $  g_\delta f_0(x)-\ell f_{0}(x)=0$, but this is absurd by the choice of $f_0$.]

	\begin{enumerate}
		\item[$(\ast)_{5.4}$] For $\delta \in U(\bar{\mathbf m}  \upharpoonright \delta)$, we say $\mathbf g \in \mathcal{G}$ is \emph{$2$-active with respect to $(\delta,\bar{\mathbf m} \upharpoonright \delta),$} when there are $\bar{\mathbf n}, \chi, h, \mathcal{B}$ witnessing $\delta \in U(\bar{\mathbf m}  \upharpoonright \delta)$ such that there are nonzero $x\in \mathbb{G}_{\bf{g}}$,
		$\alpha<\beta $ in $ \delta \setminus S$,   $f_1, f_2 \in\mathfrak{F}^{\bar{\mathbf m}  \upharpoonright \delta,{\bf g}}_{\alpha,\beta}$, and  $y_1 \neq y_2$ in $\mathbb{G}_{\mathbf g}$  such that:
		\begin{itemize}
			\item[$(a)$]  $h(f_\ell(x))\in \mathcal{Z}(H_{\mathbf m_\alpha})+ f_\ell(y_\ell)$, for $\ell=1, 2,$
			
			\item[$(b)$] $h$ maps $H_{\mathbf m_\alpha}$ (resp.  $H_{\mathbf m_\beta}$)  into $H_{\mathbf m_\alpha}$ (resp.  $H_{\mathbf m_\beta}$).
		\end{itemize}
	\end{enumerate}		
	
	\begin{enumerate}
		\item[$(\ast)_{5.5}$]  If $\bf{g}$  is  2-active with respect to $(\delta,\bar{\mathbf m} \upharpoonright \delta)$, and $\xi=\delta+1,$ then for some ${\mathbf m}_\xi$,
		$\bar{\mathbf m} \upharpoonright (\xi+1) \in \Lambda_{\xi+1}$ and
		$\delta\in U^+(\bar{\mathbf m} \upharpoonright (\xi+1))$.
	\end{enumerate}		
	[Why? Let $\alpha, \beta, f_1, f_2, y_1, y_2$  be as in $(\ast)_{5.4}$. For each $\alpha'$ with $\beta < \alpha' < \delta$, we can find $\beta' \in (\alpha',\delta) \setminus S $ and $f_3 \in \mathfrak{F}^{\bar{\mathbf m}  \upharpoonright \delta,{\bf g}}_{\alpha',\beta'}$
	such that $\mathcal{Z}(H_{\mathbf m_{\alpha'}}) \cap {\rm Im}(f_3)=0$ and  $h(f_3(x))=f_3(y_3)$ for some $y_3 \in \mathcal{G}_{\bf g}$. Then for some $\ell\in \{1, 2\}, y_3 \neq y_\ell.$ Let us suppose without loss of generality that $y_3 \neq y_1.$
	Define $f: \mathcal{G}_{\bf g} \rightarrow \mathcal{Z}(H_{\mathbf m_{\beta'}})$ as $f(z)=f_1(z) - f_3(z)$.
	Then $f \in \mathfrak{F}^{\bar{\mathbf m}  \upharpoonright \delta,{\bf g}}_{\alpha',\beta'}$ is such that  $ h(f(x))\notin \mathcal{Z}(H_{\mathbf m_{\alpha'}})+ {\rm Im}(f)$,
	$\mathcal{Z}(H_{\mathbf m_{\alpha'}}) \cap {\rm Im}(f)=0,$
	and $h$ maps $H_{\mathbf m_{\alpha'}}$ and $H_{\mathbf m_{\beta'}}$ into themselves.
	It follows that $\bf{g}$  is 1-active with respect to $(\delta,\bar{\mathbf m} \upharpoonright \delta)$, and $\xi=\delta+1,$ and we are done by $(\ast)_{5.3}$.
	]

	\begin{enumerate}
		\item[$(\ast)_{5.6}$] For $\delta \in U(\bar{\mathbf m}  \upharpoonright \delta)$, we say $\mathbf g \in \mathcal{G}$ is \emph{$3$-active with respect to $(\delta,\bar{\mathbf m} \upharpoonright \delta),$} when there are $\bar{\mathbf n}, \chi, h, \mathcal{B}$ witnessing $\delta \in U(\bar{\mathbf m}  \upharpoonright \delta)$ and  there exist
		$(y,x,z)\in\mathbb{G}_{\bf g} \times
		\mathbb{G}_	{\bf{g}} \times \mathcal{Z}(H_{{\bold m}_{\delta}})$ with $x, z \neq 0,$  $\alpha<\delta$,   $\beta\in[\alpha,\delta) \setminus S$
	and  $f\in\mathfrak{F}^{\bar{\mathbf m} \upharpoonright \delta,\bf{g}}_{\alpha,\beta}$ such that $h(f(x))=  z+f(y)$.
		Also, $h$ maps $H_{\mathbf m_\alpha}$ (resp.  $H_{\mathbf m_\beta}$)  into $H_{\mathbf m_\alpha}$ (resp.  $H_{\mathbf m_\beta}$).
	\end{enumerate}		
	\begin{enumerate}
		\item[$(\ast)_{5.7}$]  If $\bf{g}$  is  3-active with respect to $(\delta,\bar{\mathbf m} \upharpoonright \delta)$, and $\xi=\delta+1,$ then for some ${\mathbf m}_\xi$,
		$\bar{\mathbf m} \upharpoonright (\xi+1) \in \Lambda_{\xi+1}$ and
		$\delta\in U^+(\bar{\mathbf m} \upharpoonright (\xi+1))$.
	\end{enumerate}		
[Why? According to $(\ast)_{5.5}$  it is enough to show that 	$\bf{g}$  is  2-active with respect to $(\delta,\bar{\mathbf m} \upharpoonright \delta)$.
Let $\bar{\mathbf n}, \chi, h, \mathcal{B}$ witness $\delta \in U(\bar{\mathbf m}  \upharpoonright \delta)$, and fix $y, x, z, \alpha < \beta$ and $f$ as in $(\ast)_{5.6}$.
Thanks to the 3-activity assumption, we have $h(f(x))=z+f(y)$. First we claim that $y\neq 0$. Otherwise, $h(f(x))=z \in \mathcal{Z}(H_{\mathbf m_\alpha})$. Following   the choice of $h$, we have $f(x) \in \mathcal{Z}(H_{\mathbf m_\alpha})$.
But ${\rm Im}(f) \cap \mathcal{Z}(H_{\mathbf m_\alpha})=0$, hence $f(x)=0,$ which implies $x=0$, and this contradicts the assumption $x \neq 0.$ Consequently, $y\neq 0$.
Now, we set $f_1:=f$, $f_2:=2f$, $y_1:=y$ and $y_2:=2y$.
Since
$y\neq 0$, $y_1\neq y_2$. This witness $\bf{g}$  is  2-active with respect to $(\delta,\bar{\mathbf m} \upharpoonright \delta)$,
as claimed.]

\iffalse
[Why? For $n < \omega$ set $\mu_n=\prod\{m: m \in [2. n), m$ not divisible by any $p \in \mathbb{P}_{\bf c}   \}$. First we show that
	\begin{enumerate}
		\item[$(\dag)$] If $z_n \in \mathcal{Z}(H_{\mathbf m_\delta}) \setminus \mathcal{Z}(H_{\bf c})$ for $n<\omega,$ then for some sequence $\bar{s}=\langle s_n: n < \omega \rangle\in$ $^{\omega}\mathbb{Z}$, there is no limit in the $\mathbb{Z}$-adic topology to
		$\sum_{n<\omega}s_n\cdot\mu_n z_n,$
	\end{enumerate}		
	Suppose otherwise, so that for every $\bar{s}\in$ $^{\omega}\mathbb{Z}$,
	the limit $z_{\bar s}=\sum_{n<\omega}s_n\cdot \mu_n z_n $ exists in the $\mathbb{Z}$-adic topology. Now
	$ \mathcal{Z}(H_{\mathbf m_\delta})= \mathcal{Z}(H_{\mathbf c}) \oplus \bigoplus\{f_{\bf g, \xi}[\mathbb{G}_{\bf g}]: {\bf g} \in \mathcal{G}, \xi < \zeta  \}$,
	where each $f_{\bf g, \xi}$ embeds $\mathbb G_{\mathbf g}$ into $H_{\mathbf m_\delta}$, preserving the structure. So by our assumption on $z_n \in \mathcal{Z}(H_{\mathbf m_\delta}) \setminus \mathcal{Z}(H_{\bf c})$ and by projecting, without loss of generality, either
	$\{x,  z_{\bar s}   \}$ is included in  $f_{\bf g, \xi}[\mathbb{G}_{\bf g}]$, for some ${\bf g} \in \mathcal{G}, \xi < \zeta.$
	Now $z=\sum_{n<\omega}(n!)\mu_n z_n \in \hat{f_{\bf g, \xi}[\mathbb{G}_{\bf g}]}$, and by our assumption, for each $p \notin \mathbb{P}_{\bf c},$
	$ \sum_{p\leq n<\omega}(n!/p)\mu_n z_n\in \hat{f_{\bf g, \xi}[\mathbb{G}_{\bf g}]}$. But $\hat{f_{\bf g, \xi}[\mathbb{G}_{\bf g}]}$ is reduced and
	$\mathbb{P}_{\bf c}$-divisible, thus we must have $z=0,$ a contradiction.
]
	\fi

	\begin{enumerate}
		\item[$(\ast)_{5.8}$] We say $\delta \in U(\bar{\mathbf m}  \upharpoonright \delta)$ is \emph{nice} if there is no $\mathbf g \in \mathcal{G}$ which  is $\ell$-active with respect to $(\delta,\bar{\mathbf m} \upharpoonright \delta),$ for some $\ell\in\{1, 2, 3\}.$
	\end{enumerate}	
	\begin{enumerate}
		\item[$(\ast)_{5.9}$]
		Let	$\delta \in U(\bar{\mathbf m}  \upharpoonright \delta)$ be nice. Then there is a sequence $\bar h= ( h_{\mathbf g}\mid \mathbf g \in \mathcal{G} )$ such that
		for $\alpha<\delta$ large enough, and for any $f\in\bigcup\{\mathfrak{F}^{\bar{\mathbf m} \upharpoonright \delta,\bf{g}}_{\alpha,\beta}\mid \beta\in[\alpha,\delta)\}$, the following implication is valid:
		\[
		x\in\mathbb{G}_{\bf g} \Longrightarrow {g}_\delta(f(x))=f(h_{{\bf g} }(x)). ~ %(\text{ mod }  \mathcal{Z}(H_{\mathbf m_{\alpha}})).
		\]
	\end{enumerate}
	[Why? Let   $\mathbf g \in \mathcal{G}$ and $x \in \mathbb{G}_{\bf g}$. Since  $\bf g$ is not 1-active with respect to $(\delta,\bar{\mathbf m} \upharpoonright \delta),$ and following its definition,    we can find some $\alpha < \delta$ such that  ${g}_\delta(f(x)) \in \mathcal{Z}(H_{\mathbf m_{\alpha}})+{\rm Im}(f) $ for all $f\in\bigcup\{\mathfrak{F}^{\bar{\mathbf m} \upharpoonright \delta,\bf{g}}_{\alpha,\beta}\mid \beta\in[\alpha,\delta)\}$.
Hence there is some $z \in \mathcal{Z}(H_{\mathbf m_{\alpha}})$ and $y \in \mathbb{G}_{\bf g} $ so that ${g}_\delta(f(x))=z+f(y)$. Recall that $\bf g$ is not 3-active. This forces $z=0.$
Applying this in the previous formula, gives us ${g}_\delta(f(x))=f(y).$ Since both of $g_\delta$ and $f$ are injective, $y$ is uniquely determined via $h$ and ${\bf g} $, so let $y=h_{{\bf g} }(x)$. Then $h_{{\bf g} }$ is as required.]
	
	Let $\alpha_* < \delta$ be such that $(\ast)_{5.9}$ holds for all $\alpha \geq \alpha_*.$
	\begin{enumerate}
		\item[$(\ast)_{5.10}$]
		If	$\delta \in U(\bar{\mathbf m}  \upharpoonright \delta)$ is nice, then:
		\begin{enumerate}
			\item For all $\alpha<\delta$ and $f\in\bigcup\{\mathfrak{F}^{\bar{\mathbf m} \upharpoonright \delta,\bf{g}}_{\alpha,\beta}\mid \beta\in[\alpha,\delta)\}$, we have ${g}_\delta(f(x))=f(h_{{\bf g} }(x))$ for $x\in\mathbb{G}_{\bf g}$.
			\item Let $\mathbf g \in \mathcal{G}$
			and $f$ an embedding from $(\mathbb{G}_{\mathbf g}, F_{\mathbf g}^b)_{b \in L_{\mathbf c}}$ into $(\mathcal{Z}(H_{\mathbf m_\delta}), F_{\mathbf m_\delta}^b)_{b \in L_{\mathbf c}}$ and let $x \in \mathbb{G}_{\mathbf g}.$
			Then there is a sequence $\langle b_i, s_i\mid i<n  \rangle$  with $b_i \in  L_{\mathbf c}$
			and $s_i \in \mathbb{Z}\setminus \{0\}$ such that
			$
			g_\delta(f(x)) - \sum_{i=1}^n s_i F_{\mathbf m_\delta}^{b_i}(f(x)) \in \mathcal{Z}(H_{\mathbf m_{\alpha_*}}).
			$
		\end{enumerate}
	\end{enumerate}
	[Why? (a):  Suppose not. Let $\alpha < \delta$ and $f_1\in \mathfrak{F}^{\bar{\mathbf m} \upharpoonright \delta,\bf{g}}_{\alpha, \beta}$ be a counterexample, where $\beta > \alpha, \alpha_*$. Thus for some $x\in\mathbb{G}_{\bf g}$, we have ${g}_\delta(f_1(x))\neq f_1(h_{{\bf g} }(x))$.  We may further assume that $\alpha_*> \alpha$ (as $(\ast)_{5.9}$ works for any ordinal in $(\alpha_*, \delta)$ as well).
	Let $ f_2\in \mathfrak{F}^{\bar{\mathbf m} \upharpoonright \delta,\bf{g}}_{\alpha_*,\beta}$ be such that $\mathcal{Z}(H_{\mathbf m_{\alpha_*}}) \cap {\rm Im}(f_2)=0$. It  turns out that  $f:=f_1+f_2 \in \mathfrak{F}^{\bar{\mathbf m} \upharpoonright \delta,\bf{g}}_{\alpha_*,\beta}$, thus by $(\ast)_{5.9}$,
	$ {g}_\delta(f(x))=f(h_{{\bf g} }(x)).$
	This means that
 ${g}_\delta(f_1(x)) + {g}_\delta(f_2(x)) = f_1(h_{{\bf g} }(x)) + f_2(h_{{\bf g} }(x)).$  Hence as ${g}_\delta(f_2(x)) =  f_2(h_{{\bf g} }(x))$,
	we  have ${g}_\delta(f_1(x))= f_1(h_{{\bf g} }(x)),$  a contradiction.

	(b): This is true as $g_\delta(f(x)) \in \langle \mathcal{Z}(H_{\mathbf m_{\alpha_*}}) \cup \{ F_{\mathbf m_\delta}^b(f(x))\mid b \in L_{\mathbf c} \} \rangle.$]
	
	\begin{enumerate}
		\item[$(\ast)_{5.11}$]
		Let	$\delta \in U(\bar{\mathbf m}  \upharpoonright \delta)$ be nice.
		Then there is a sequence $( b^*_i, s^*_i\mid i<n  )$  where $b^*_i \in  L_{\mathbf c}$ and
		$s^*_i \in \mathbb{Z}\setminus \{0\}$ such that for all $\mathbf g \in \mathcal{G}$, and any embedding $f:  (\mathbb{G}_{\mathbf g}, F_{\mathbf g}^b)_{b \in L_{\mathbf c}}\to  (\mathcal{Z}(H_{\mathbf m_\delta}), F_{\mathbf m_\delta}^b)_{b \in L_{\mathbf c}}$ we have
		$$x\in\mathbb{G}_{\bf g} \Rightarrow g_\delta(f(x)) - \sum_{i<n}s^*_i F_{\mathbf m_\delta}^{b^*_i}(f(x)) \in \mathcal{Z}(H_{\mathbf m_{\alpha_*}}).$$
	\end{enumerate}
	[Why? Fix $\mathbf g_*  \in \mathcal{G}_{\bf g}$  and
	$x_* \in \mathcal{G}_{\bf g}$. Due to $(\ast)_{5.10}(b)$ applied to  $\mathbf g_* \rest_{{\rm cl}\{x_*\}},$ we can find a sequence $( b^*_i, s^*_i\mid i<n  )$
	as there.
	Let $\mathbf g \in \mathcal{G}$ and suppose $f$ is an embedding from $(\mathbb{G}_{\mathbf g}, F_{\mathbf g}^b)_{b \in L_{\mathbf c}}$ into $(\mathcal{Z}(H_{\mathbf m_\delta}), F_{\mathbf m_\delta}^b)_{b \in L_{\mathbf c}}$ and $x \in \mathbb{G}_{\mathbf g}.$ By replacing $\mathbf g$ by $\mathbf g \rest_{{\rm cl}\{x\}} $, we may assume that $\mathbb{G}_{\mathbf g}= \operatorname{cl}\{x\}$ is one-generated (see Definition \ref{def3}(4)(b)). Let
	$\phi: \operatorname{cl}\{x_*\} \rightarrow \operatorname{cl}\{x\}$ be such that $\phi(x_*)=x,$ and set $\tilde f:=f \circ \phi$. Note that $\phi$ is an embedding from the structure
	$(\mathbb{G}_{\mathbf g_* \rest_{{\rm cl}\{x_*\}}}, F_{\mathbf g_* \rest_{{\rm cl}\{x_*\}}}^b)_{b \in L_{\mathbf c}}$ into $(\mathbb{G}_{\mathbf g \rest_{{\rm cl}\{x\}}}, F_{\mathbf g \rest_{{\rm cl}\{x\}}}^b)_{b \in L_{\mathbf c}}$ and $\tilde f$ is an
	embedding from the structure
	$(\mathbb{G}_{\mathbf g_* \rest_{{\rm cl}\{x_*\}}}, F_{\mathbf g_* \rest_{{\rm cl}\{x_*\}}}^b)_{b \in L_{\mathbf c}}$ into $(\mathcal{Z}(H_{\mathbf m_\delta}), F_{\mathbf m_\delta}^b)_{b \in L_{\mathbf c}}$ and $x \in \mathbb{G}_{\mathbf g}.$
	Following our assumption, it implies that
	$g_\delta(\tilde f(z_e)) - \sum_{i=1}^n s^*_i F_{\mathbf m_\delta}^{b^*_i}(\tilde f(z_e)) \in \mathcal{Z}(H_{\mathbf m_{\alpha_*}}).$
	Recall that $f, \tilde f$ and $\phi$ are embeddings that respect the structures. So,
		$$F_{\mathbf m_\delta}^{b^*_i}(\tilde f(z_e))\stackrel{}= \tilde f (F_{\mathbf g_* \rest_{{\rm cl}\{x_*\}}}^{b^*_i}(z_e)) =f (\phi(F_{\mathbf g_* \rest_{{\rm cl}\{x_*\}}}^{b^*_i}(z_e)))\stackrel{}= f(F_{\mathbf g \rest_{{\rm cl}\{x\}}}^{b^*_i}(x))\stackrel{}= F_{\mathbf g \rest_{{\rm cl}\{x\}}}^{b^*_i}(f(x)).$$
	But note that $F_{\mathbf g \rest_{{\rm cl}\{x\}}}^{b^*_i}(f(x))=F_{\mathbf g}^{b^*_i}(f(x)),$  hence
	\[
	g_\delta(f(x)) - \sum_{i=1}^n s^*_i F_{\mathbf m_\delta}^{b^*_i}(f(x)) = g_\delta(\tilde f(z_e))- \sum_{i=1}^n s^*_i F_{\mathbf m_\delta}^{b^*_i}(\tilde f(z_e)) \in \mathcal{Z}(H_{\mathbf m_{\alpha_*}}).\space\ ]
	\]
	$(\ast)_6$: Let us derive the desired presentation in the above nice case.
	Recall that each element of $\mathcal{Z}(H_{\mathbf m_{\delta}})$ is of the form $f(x)$ for some $f\in\mathfrak{F}^{\bar{\mathbf m} \upharpoonright \delta,\bf{g}}_{\alpha, \beta}$ and $x\in \mathbb{G}_{\mathbf g}$, and also recall that $h_{\mathbf c}: H_{\mathbf c} \rightarrow N_{\mathbf c}$ is an epimorphism  with $\Ker(h_{\mathbf c})=\mathcal{Z}(H_{\mathbf c})$. These yield  an automorphism $\pi \in \aut(N_{\mathbf c})$
	via the assignment
	$ h_{\mathbf c}(f(x))\mapsto h_{\mathbf c}(g_\delta(f(x))).$  In view of $(\ast)_{5.11}$ we set $F_{\mathbf c}^{{\bar s^*}, {\bar b^*}} = \sum_{i=1}^n s^*_i F_{\mathbf c}^{b^*_i}$.
	Thanks to clause (1) from our assumption, for some $b \in L_{\mathbf c}$ one has $ F_{\mathbf m_\delta}^b \rest_{H_{\mathbf c}} = F_{\mathbf c}^{{\bar s^*}, {\bar b^*}}$.  By the way we extended the functions, $F_{\mathbf m_\delta}^b= F_{\mathbf m_\delta}^{{\bar s^*}, {\bar b^*}}$, and following its definition, we have
	$\pi(h_{\mathbf c}(f(x)))=h_{\mathbf c}(F_{\mathbf m_\delta}^b(f(x)))$  for all $f$ and $x$ as above.
	Let $g_1:= (F_{\mathbf m_\delta}^b)^{-1} \circ g_\delta, $ and recall that $g_1 \rest_{\mathcal{Z}(H_{\mathbf m_\delta})}=\id$
	and  $g_1(t) \in t\cdot \mathcal{Z}(H_{\mathbf m_\delta})$ for all $t \in H_{\mathbf m_\delta}$. According to Proposition \ref{2b.16}, there
	is  some $f \in \Hom(N_{\mathbf c}, \mathcal{Z}(H_{\mathbf m_\delta}))$ so that $g_1=g_{f}$.
Due to clause (1) of the theorem and in the light of Lemma \ref{tri} we observe that $f \in \Hom(N_{\mathbf c},  \mathcal{Z} (H_{\mathbf c})).$ It follows that
	$h \restriction \delta=g_\delta= F_{\mathbf m_\delta}^b \circ g_1=F_{\mathbf m_\delta}^b \circ g_f$. This is a contradiction to
	$(\ast)_1(d+g)$.

	In sum, we have proved that
	$ \Omega_\gamma\neq \emptyset,$  completing the proof of 	$(\ast)_6$. Recall from clause $(\ast)_4$ that this completes
	the proof of the theorem.
\end{proof}
\iffalse
The following is a corollary of the above theorem.
\begin{corollary}\label{5h.17}
	Assume $\bold V = \bold L$. Let $L$ be a group such that for some group
	$H$ and some normal subgroup $N$ of $L$ we have
	\begin{itemize}
		\item $|H| > |L|$ and $F: L \cong
		\aut(H)$,
		\item there exists an exact sequence $$\xymatrix{
			&&&&0\ar[r] &\mathcal{Z}(H)\ar[r] &H\ar[r]_{ h}&N\ar[r] &0
			&&&}$$
		such that $N \cong F(N)={\rm Inn}(H) \subseteq {\rm Aut}(H)$,
		
		%\item ${\rm Hom}(N,\Bbb Z)=0$.
	\end{itemize}
	Then there are arbitrary large groups $H'$
	such that $L \cong \aut(H')$.
\end{corollary}
\begin{proof}
	Define $\mathbf c \in C_{\rm{aut}}$ as follows:
	\begin{enumerate}
		\item $L_{\mathbf c}= L$,
		
		\item $H_{\mathbf c} = H$,
		
		\item $N_{\mathbf c}=N$,
		\item $F_{\mathbf c}=F,$
		\item $h_{\mathbf c}=h$ and $h^*_{\mathbf c}: N \rightarrow H$ is defined naturally,  so that $h \circ h^*_{\mathbf c}=\id,$
		
		\item $Q^{\bar s}_{\mathbf c}$ is defined naturally as in Definition \ref{2b.4}.
	\end{enumerate}
	The assumption $|H| > |L|$ implies, by the proof of \ref{lem3}, that ${\rm Hom}(N,\Bbb Z)=0$, hence by Theorem  \ref{5h.11},
	for any $\lambda$ large enough regular, we can find some $H'$ of size $\lambda$
\end{proof}
\fi
\begin{lemma}
	\label{t11} Assume
	$\mathbf c\in\mathbf C^4_{\rm{aut}}$
	is such that  $|H_{\mathbf c}| > 2^{|L_ {\mathbf c}|+ \aleph_0}$.  Then we can extend $\mathbf c$
	to some $(\mathbf c, \mathcal{G}) \in \mathbf C^6_{\rm{aut}}$.
\end{lemma}
\begin{proof}
	It suffices to define $\mathcal{G}.$ First,
	we say $\bf g$ is \emph{one-generated}, provided  there is some $x \in \mathcal{Z}(H_{\bf c})$ so that $\mathbb{G}_{\bf g} / ( H^*_{\bf c} \cap \mathcal{Z}(H_{\bf c}))$
	is the smallest pure subgroup of $ \mathcal{Z}(H_{\bf c}) / ( H^*_{\bf c} \cap \mathcal{Z}(H_{\bf c}))$
	to which $x$ belongs and that is closed under $F_{\bf g}^b$'s for $b \in L_{\bf c}.$
Let $\mathcal{G} $ consist of all $\mathbf g$ such that:
	\begin{enumerate}
		\item ${\bf g}=(\mathbb{G}_{\bf g}, (F_{\bf {g}}^b)_{b \in L_{\bf c}})$,  where $\mathbb{G}_{\bf g} \subseteq \mathcal{Z}(H_{\bf c}),$
		
		\item $\mathbb{G}_{\bf g}$ is a torsion-free subgroup such that $\mathbb{G}_{\bf g} \cap H^*_{\bf c}=0,$
		
		\item $\bf g$ is one-generated,
		
		\item $|\mathbb{G}_{\bf g}| \leq |L_{\bf c}|+\aleph_0$.
	\end{enumerate}
	\begin{enumerate}
		\item[$(\ast)_{1}$] Denoting the isomorphic structures equivalence relation with $\cong$, we have:
		\begin{enumerate}
			\item $|\mathcal{G}/\cong| \leq 2^{|L_{\bf c}|+\aleph_0}$.
			
			\item If $\mathcal{G} \neq \emptyset,$ then $(\bf c, \mathcal{G}) \in \mathbf C^6_{\rm{aut}}.$
			
			\item If $\bf g \in \mathcal{G}$, then $\Hom(N_{\bf c}, \mathbb{G}_{\bf g})=0.$
		\end{enumerate}
	\end{enumerate}
	
	\begin{enumerate}
		\item[$(\ast)_{2}$] Assume  $|H_{\mathbf c}| > 2^{|L_ {\mathbf c}|+ \aleph_0}$. Then $\mathcal{G} \neq \emptyset$.
	\end{enumerate}
	[Why?
First, note that  $|\mathcal{Z}(H_{\mathbf c})|\geq 2^{|L_ {\mathbf c}|+\aleph_0}$. From this, we can find a sequence ${(  x_\alpha\mid \alpha < 2^{|L_ {\mathbf c}|+\aleph_0}    )}$
	of distinct elements of  $\mathcal{Z}(H_{\mathbf c}) \setminus H^*_{\bf c}$. For each $\alpha$, let $\mathbb{G}_\alpha$ be minimal such that:
	\begin{enumerate}[(a)]
		\item $\{x_\alpha\} \cup ( H^*_{\bf c} \cap \mathcal{Z}(H_{\bf c})) \subseteq \mathbb{G}_\alpha,$
		
		\item $\mathbb{G}_\alpha$ is closed under the action of $F_{\bf c}^b$, for $b \in L_{\bf c}$,
		
		\item $\mathbb{G}_\alpha/ ( H^*_{\bf c} \cap \mathcal{Z}(H_{\bf c}))$ is a pure subgroup of $ \mathcal{Z}(H_{\bf c}) / ( H^*_{\bf c} \cap \mathcal{Z}(H_{\bf c}))$,
		
		\item $|\mathbb{G}_\alpha| \leq |L_{\bf c}|+\aleph_0.$
	\end{enumerate}
	For each $\alpha$ let $( x_{\alpha, \ell}\mid \ell <   |\mathbb{G}_\alpha| )$ enumerate $\mathbb{G}_\alpha$ so that the elements of $ H^*_{\bf c} \cap \mathcal{Z}(H_{\bf c})$ are enumerated first.
	Then for some $\alpha < \beta < 2^{|L_ {\mathbf c}|+\aleph_0}$ we have $|\mathbb{G}_\alpha|=|\mathbb{G}_\beta|$
	and $\{(x_{\alpha, \ell}, x_{\beta, \ell})\mid \ell < |\mathbb{G}_\alpha|  \}$ is an isomorphism from $\mathbb{G}_\alpha$
	onto $\mathbb{G}_\beta$ which is identity on $ H^*_{\bf c} \cap \mathcal{Z}(H_{\bf c})$
	and commutes under $F_{\bf c}^b$, for all $b \in L_{\bf c}$. In order to define $\bf g $, we set
	$\mathbb{G}_{\bf g}=\{x_{\alpha, \ell}- x_{\beta, \ell}\mid \ell < |\mathbb{G}_\alpha|  \}$ and define $F_{\bf g}^b:=F_{\bf c}^b \rest_{\mathbb{G}_{\bf g}}.$
	It turns out that $\bf g \in \mathcal{G}$.]
	
	The lemma follows.
\end{proof}

\begin{corollary}\label{a25}  Assume G\"{o}del's axiom of constructibility $\bf V=\bf L$, and  let $L$ be a group. Then the following are equivalent:
	\begin{itemize}
		\item[$(a)$] For every  cardinal $\lambda>2^{|L |+ {\aleph_0}}$  as in Theorem \ref{5h.11}, there is a group $H$ of cardinality  $\lambda$ such that $\aut(H)\cong L$.
		\iffalse		\item[$(b)$] there is a group $H$ of cardinality $\lambda$
		and $H'\lhd H$ of cardinality  $\leq|L |+2^{\aleph_0}$ such that $\aut(H)\cong L$ and $\frac{H}{H'}$ is $\lambda$-free.
		\fi	
		\item[$(b)$]
		There is some 	$\mathbf c\in\mathbf C^+_{\rm{aut}}$
		such that $L_ \mathbf c\cong L$, $\aut(H_{\mathbf c})\cong L$ and $|\mathcal{Z}(H_{\mathbf c})| > 2^{|L |+\aleph_0}$.
		\iffalse
		\item[$(d)$] Similar to (c) for $\mathbf d$
		$\mathbb{Q}_{\mathbb{P}}$-module where $\mathbb{P}:=\mathbb{P}_{(L,H,F)}$.\fi
	\end{itemize}
	%2) Concerning the situation of Theorem \ref{5h.1b} (resp. Theorem \ref{5h.1c}), there is a corresponding statement as in part (1).
\end{corollary}

\begin{proof}

	$(a)\Rightarrow (b)$:
	This is clear, as we can define $\mathbf c\in\mathbf C^+_{\rm{aut}}$ by $\mathbf c=(L_{\mathbf c}, H_{\mathbf c}, F_{\mathbf c})=(L, H, F)$,
	where $F: L \cong \aut(H)$ is an isomorphism.
	
	$(b)\Rightarrow (a)$:  Let $\mathbf c\in\mathbf C^+_{\rm{aut}}$
	be such that $L_ \mathbf c\cong L$ and $|H_{\mathbf c}| > 2^{|L|+\aleph_0}$. Let also  $\lambda>2^{|L |+ {\aleph_0}}$ be as in Theorem \ref{5h.11}.  We combine Lemma   \ref{a22}(3) along with Lemma~\ref{t11}, and extend $\mathbf c$
	to some $(\mathbf c, \mathcal{G}) \in \mathbf C^6_{\rm{aut}}$.
	We are going to show that $\Hom(N_{\mathbf c}, \mathbb{Z})=0$.
	Indeed, this follows from the fact that the ${\mathbf c}$ we construct is auto-rigid and the center of $H_{\mathbf c}$ has size bigger than the size of $L$.
	Thanks to Lemma \ref{lem3}, with ${\mathbf m}={\mathbf c}$, we see $\Hom(N_{\mathbf c}, \mathbb{Z})=0$. In fact, following the
	previous proof, $\Hom(N_{\bf c}, \mathbb{G}_{\bf g})= 0,$ for all $\bf g \in \mathcal{G}.$
	Theorem \ref{5h.11}  gives us
	an ${\mathbf m}$
	so that $|H_{\mathbf m}| =\lambda$ and
	$\aut(H_{\mathbf m})=\{F_{\mathbf m}^b \circ g_f\mid
	b\in L_\mathbf{c}, f \in \Hom(N_{\bold c},\mathcal{Z}(H_{\bold m}))\}.$
	We note that for any $b, f$ as above,
	$(F_{\mathbf m}^b \circ g_f) \rest_{H_{\mathbf c}} \in \aut(H_{\mathbf c})$,
	hence for some $d \in L_{\mathbf c},$ we have $(F_{\mathbf m}^b \circ g_f) \rest_{H_{\mathbf c}} =F_{\mathbf c}^d$.
	Then $F_{\mathbf m}^b \circ g_f =F_{\mathbf m}^d$. Thus
$
	\aut(H_{\mathbf m}) \subseteq \{F_{\mathbf m}^d\mid  d \in L_{\mathbf c}    \} \subseteq \aut(H_{\mathbf m}),
$
	so $\aut(H_{\mathbf m}) \cong L_{\mathbf c} \cong L.$
\end{proof}

Similarly, we can prove the following.
\begin{corollary}\label{a299}  Assume G\"{o}del's axiom of constructibility $\bf V=\bf L,$ and let $L$ be a group.  Then the following are equivalent:
	\begin{itemize}
		\item[$(a)$] For every $\lambda>2^{|L |+ {\aleph_0}}$  as in Theorem \ref{5h.11}, there is a group $H$ of cardinality  $\lambda$ such that $\aut(H)\cong L$,
		and for some $H_* \subseteq \mathcal{Z}(H)$ of cardinality $<\!\lambda$, $ \mathcal{Z}(H)/ H_*$ is a $\lambda$-free abelian group.
		
		\item[$(b)$]
		There is 	$\mathbf c\in\mathbf C^+_{\rm{aut}}$
		such that $L_ \mathbf c\cong L$, $\aut(H_{\bf c})\cong L$ and $|\mathcal{Z}(H_{\mathbf c})| > 2^{|L |+\aleph_0}$, and there is $\bf g \in \mathcal{G}$ such that
		$\mathbb{G}_{\bf g}$ is a free abelian group.
	\end{itemize}
	%2) Concerning the situation of Theorem \ref{5h.1b} (resp. Theorem \ref{5h.1c}), there is a corresponding statement as in part (1).
\end{corollary}

\end{document}